\documentclass{amsart}

\pdfoutput=1
\usepackage{amssymb}
\usepackage{amsthm}
\usepackage{amsmath}
\usepackage{graphicx}
\usepackage{subcaption}
\usepackage[color=white]{todonotes}
\usepackage{float}
\usepackage{enumitem}

\title{A generalized finite element method for linear thermoelasticity}
\author[Axel M{\aa}lqvist]{Axel M{\aa}lqvist\textsuperscript{1,2}} 
\author[Anna Persson]{Anna Persson\textsuperscript{1}}

\newtheorem{mydef}{Definition}[section]
\newtheorem{mythm}[mydef]{Theorem}
\newtheorem{mylemma}[mydef]{Lemma}
\theoremstyle{remark}
\newtheorem{myremark}[mydef]{Remark}
\newtheorem*{myassump}{Assumptions}
\numberwithin{equation}{section}

\DeclareMathOperator*{\esssup}{ess\,sup}
\DeclareMathOperator*{\essinf}{ess\,inf}
\DeclareMathOperator*{\Span}{span}

\DeclareMathOperator{\diam}{diam}
\DeclareMathOperator{\card}{card}

\newcommand{\ms}{\mathrm{ms}}
\newcommand{\f}{\mathrm{f}}
\newcommand{\ddt}{\bar\partial_t}

\makeatletter
\newcommand{\enorm}{\@ifstar\@enorms\@enorm}
\newcommand{\@enorms}[1]{%
  \left|\mkern-1.5mu\left|\mkern-1.5mu\left|
   #1
  \right|\mkern-1.5mu\right|\mkern-1.5mu\right|
}
\newcommand{\@enorm}[2][]{%
  \mathopen{#1|\mkern-1.5mu#1|\mkern-1.5mu#1|}
  #2
  \mathclose{#1|\mkern-1.5mu#1|\mkern-1.5mu#1|}
}
\makeatother

\usepackage{hyperref}

\begin{document}
\begin{abstract}
	We propose and analyze a generalized finite element method designed for linear quasistatic thermoelastic systems with spatial multiscale coefficients. The method is based on the local orthogonal decomposition technique introduced by M{\aa}lqvist and Peterseim (Math. Comp., 83(290): 2583--2603, 2014). We prove convergence of optimal order, independent of the derivatives of the coefficients, in the spatial $H^1$-norm. The theoretical results are confirmed by numerical examples. 
\end{abstract}
\maketitle

\footnotetext[1]{Department of Mathematical Sciences, Chalmers University of Technology and University of Gothenburg, SE-412 96 G\"{o}teborg, Sweden.}
\footnotetext[2]{Supported by the Swedish Research Council and the Swedish Foundation for Strategic Research.}

\section{Introduction}\label{sec:intro}
In many applications the expansion and contraction of a material exposed to temperature changes are of great importance. To model this phenomenon a system consisting of an elasticity equation describing the displacement coupled with an equation for the temperature is used, see, e.g., \cite{Biot56}. The full system consists of a hyperbolic elasticity equation coupled with a parabolic equation for the temperature, see \cite{Dafermos68} for a comprehensive treatment of this formulation. If the inertia effects are negligible, the hyperbolic term in the elasticity equation can be removed. This leads to an elliptic-parabolic system, often referred to as \textit{quasistatic}. This formulation is discussed in, for instance, \cite{Showalter00, Zenisek84-2}. In some settings it is justified to also remove the parabolic term, which leads to an elliptic-elliptic system, see, e.g., \cite{Showalter00, Zenisek84-2}. Since the thermoelastic problem is formally equivalent to the system describing poroelasticity, several papers on this equation are also relevant, see, e.g., \cite{Biot41, Zenisek84-1}.

In this paper we study the quasistatic case. Existence and uniqueness of a solution to this system are discussed in \cite{Showalter00} within the framework of linear degenerate evolution equations in Hilbert spaces. It is also shown that this system is essentially of parabolic type. Existence and uniqueness are also treated in \cite{Zenisek84-2} (only two-dimensional problems) and in \cite{Xu96,Shi92} some results on the thermoelastic contact problem are presented. The classical finite element method for the thermoelastic system is analyzed in \cite{Ern09,Zenisek84-2}, where convergence rates of optimal order are derived for problems with solution in $H^2$ or higher. 

When the elastic medium of interest is strongly heterogeneous, like composite materials, the coefficients are highly varying and oscillating. Commonly, such coefficients are said to have \textit{multiscale} features. For these problems classical polynomial finite elements, as in \cite{Ern09,Zenisek84-2}, fail to approximate the solution well unless the mesh width resolves the data variations. This is due to the fact that a priori bounds of the error depend on (at least) the spatial $H^2$-norm of the solution. Since this norm depends on the derivative of the diffusion coefficient, it is of order $\epsilon^{-1}$ if the coefficient oscillates with frequency $\epsilon^{-1}$. To overcome this difficulty, several numerical methods have been proposed, see for instance \cite{Babuska83, Babuska11, Larson07, Malqvist14, Hughes98}.  

In this paper we suggest a generalized finite element method based on the techniques introduced in \cite{Malqvist14}, often referred to as \textit{local orthogonal decomposition}. This method builds on ideas from the variational multiscale method \cite{Hughes98,Larson07}, where the solution space is split into a coarse and a fine part. The coarse space is modified such that the basis functions contain information from the diffusion coefficient and have support on small patches. With this approach the basis functions have good approximation properties locally. In \cite{Malqvist14} the technique is applied to elliptic problems with an arbitrary positive and bounded diffusion coefficient. One of the main advantages is that no assumptions on scale separation or periodicity of the coefficient are needed. Recently, this technique has been applied to several other problems, for instance, semilinear elliptic equations \cite{HMP14}, boundary value problems \cite{Henning14}, eigenvalue problems \cite{Malqvist13}, linear and semilinear parabolic equations \cite{MP15}, and the linear wave equation \cite{Abdulle14}.

The method we propose in this paper uses generalized finite element spaces similar to those used \cite{Malqvist14} and \cite{HP16}, together with a correction building on the ideas in \cite{Henning14, Larson07}. We prove convergence of optimal order that does not depend on the derivatives of the coefficients. We emphasize that by avoiding these derivatives, the a priori bound does not contain any constant of order $\epsilon^{-1}$, although coefficients are highly varying.

In Section~\ref{sec:problem} we formulate the problem of interest, in Section~\ref{sec:num} we first recall the classical finite element method for thermoelasticity and then we define the new generalized finite element method. In Section~\ref{sec:loc} we perform a localization of the basis functions and in Section~\ref{sec:error} we analyze the error. Finally, in Section~\ref{sec:ex} we present some numerical results.

\section{Problem formulation}\label{sec:problem}
Let $\Omega \subseteq \mathbb{R}^d$, $d=2,3$, be a polygonal/polyhedral domain describing the reference configuration of an elastic body. For a given time $T>0$ we let $u: [0,T]\times \Omega \rightarrow \mathbb{R}^d$ denote the displacement field and $\theta: [0,T]\times \Omega \rightarrow \mathbb{R}$ the temperature. To impose Dirichlet and Neumann boundary conditions, we let $\Gamma^u_D$ and $\Gamma^u_N$ denote two disjoint segments of the boundary such that $\Gamma := \partial \Omega = \Gamma^u_D \cup \Gamma^u_N$. The segments $\Gamma^\theta_D$ and $\Gamma^\theta_N$ are defined similarly. 

We use $(\cdot,\cdot)$ to denote the inner product in $L_2(\Omega)$ and $\|\cdot\|$ for the corresponding norm. Let $H^1(\Omega)$ denote the classical Sobolev space with norm $\|v\|^2_{H^1(\Omega)}=\|v\|^2+\|\nabla v\|^2$ and let $H^{-1}(\Omega)$ denote the dual space to $H^1$. Furthermore, we adopt the notation $L_p([0,T];X)$ for the Bochner space with the norm
\begin{align*}
\|v\|_{L_p([0,T];X)} &= \Big(\int_0^T \|v\|_X^p \, \mathrm{dt} \Big)^{1/p}, \quad 1\leq p<\infty,\\
\|v\|_{L_\infty([0,T];X)} &= \esssup_{0\leq t \leq T} \|v\|_X,
\end{align*}
where $X$ is a Banach space equipped with the norm $\| \cdot\|_X$. The notation $v\in H^1(0,T;X)$ is used to denote $v,\dot{v}\in L_2(0,T;X)$. The dependence on the interval  $[0,T]$ and the domain $\Omega$ is frequently suppressed and we write, for instance, $L_2(L_2)$ for $L_2([0,T];L_2(\Omega))$. We also define the following subspaces of $H^1$
\begin{align*}
V^1 := \{v \in (H^1(\Omega))^d: v=0 \ \text{on} \ \Gamma^u_D\}, \quad V^2 := \{v \in H^1(\Omega): v=0 \ \text{on} \ \Gamma^\theta_D\}.
\end{align*}

Under the assumption that the displacement gradients are small, the (linearized) strain tensor is given by 
\begin{align*}
\varepsilon(u)=\frac{1}{2}(\nabla u + \nabla u^\intercal).
\end{align*}
Assuming further that the material is isotropic, Hooke's law gives the (total) stress tensor, see e.g. \cite{Shi92} and the references therein, 
\begin{align*}
\bar\sigma = 2\mu \varepsilon(u) + \lambda (\nabla \cdot u) I - \alpha \theta I,
\end{align*}
where $I$ is the $d$-dimensional identity matrix, $\alpha$ is the thermal expansion coefficient, and $\mu$ and $\lambda$ are the so called Lam\'{e} coefficients given by
\begin{align*}
\mu = \frac{E}{2(1+\nu)}, \quad \lambda = \frac{E\nu}{(1+\nu)(1-2\nu)},
\end{align*}
where $E$ denotes Young's elastic modulus and $\nu$ denotes Poisson's ratio. The materials of interest are strongly heterogeneous which implies that $\alpha$, $\mu$, and $\lambda$ are rapidly varying in space. 

The linear quasistatic thermoelastic problem takes the form 
\begin{alignat}{2}
	-\nabla \cdot (2\mu \varepsilon(u) + \lambda \nabla \cdot u I -\alpha \theta I) &= f,& \quad &\text{in } (0,T] \times \Omega, \label{thermo1}  \\
	\dot\theta-\nabla\cdot \kappa \nabla \theta + \alpha\nabla \cdot \dot u&=g,& &\text{in } (0,T] \times \Omega, \label{thermo2}\\
	u &= 0,& & \text{in } (0,T]\times\Gamma^u_D,\label{bd1}\\
	\bar\sigma \cdot n &= 0,& & \text{in } (0,T]\times \Gamma^u_N.\label{bd2}\\
	\theta &= 0,& & \text{on } (0,T]\times \Gamma^\theta_D,\label{bd3}\\
	\nabla \theta \cdot n &= 0,& & \text{on } (0,T]\times \Gamma^\theta_N.\label{bd4}\\
	\theta(0) &=\theta_0,& & \text{in } \Omega, \label{initial}
\end{alignat}
where $\kappa$ is the heat conductivity parameter, which is assumed to be rapidly varying in space. 

\begin{myremark}
	For simplicity we have assumed homogeneous boundary data \eqref{bd1}-\eqref{bd4}. However, using techniques similar to the ones used in \cite{Henning14,HP16} the analysis in this paper can be extended to non-homogeneous situations.  
\end{myremark}

\begin{myassump} We make the following assumptions on the data
	\begin{enumerate}[label=(A\arabic*)]
		\item\label{A1} $\kappa \in L_\infty(\Omega,\mathbb{R}^{d \times d})$, symmetric,
		\begin{align*} 
		0<\kappa_1 :=\essinf_{x \in \Omega} \inf_{v \in \mathbb{R}^d\setminus \{0\}} \frac{\kappa(x) v \cdot v}{v \cdot v}, \quad
		\infty > \kappa_2:=\esssup_{x \in \Omega} \sup_{v \in \mathbb{R}^d\setminus \{0\}} \frac{\kappa(x) v \cdot v}{v \cdot v},
		\end{align*}
		\item \label{A2} $\mu,\lambda,\alpha \in L_\infty(\Omega, \mathbb{R})$, and
		\begin{align*} 
		0 < \mu_1 := \essinf_{x \in \Omega} \mu(x) \leq \esssup_{x \in \Omega} \mu(x) =: \mu_2 < \infty.
		\end{align*}
		Similarly, the constants $\lambda_1, \lambda_2, \alpha_1$, and $\alpha_2$ are used to denote the corresponding upper and lower bounds for $\lambda$ and $\alpha$.
		\item \label{A3} $f,\dot f \in L_\infty(L_2), \ddot f \in L_\infty(H^{-1})$, $g\in L_\infty(L_2)$, $\dot g\in L_\infty(H^{-1})$, and $\theta_0 \in V^2$.
	\end{enumerate}
\end{myassump}

To pose a variational form we multiply the equations \eqref{thermo1} and \eqref{thermo2} with test functions from $V^1$ and $V^2$ and using Green's formula together with the boundary conditions \eqref{bd1}-\eqref{bd4} we arrive at the following weak formulation \cite{Ern09}. Find $u(t,\cdot)\in V^1$ and $\theta(t,\cdot) \in V^2$, such that,
\begin{alignat}{2}
	(\sigma(u):\varepsilon(v_1))- (\alpha \theta, \nabla \cdot v_1) &= (f,v_1),& \quad &\forall v_1 \in V^1,\label{weak1}\\
	(\dot\theta,v_2) + (\kappa \nabla \theta, \nabla v_2) + (\alpha \nabla \cdot \dot u,v_2)&= (g,v_2),& &\forall v_2 \in V^2 \label{weak2},
\end{alignat}
and the initial value $\theta(0,\cdot) = \theta_0$ is satisfied. Here we use $\sigma$ to denote the effective stress tensor $\sigma(u):=2\mu \varepsilon(u) + \lambda (\nabla \cdot u) I$ and we use $:$ to denote the Frobenius inner product of matrices. Using Korn's inequality we have the following bounds, see, e.g., \cite{Ciarlet88},
\begin{align}\label{sigma-bounds}
c_\sigma\|v_1\|^2_{H^1}\leq (\sigma(v_1):\epsilon(v_1)) \leq C_\sigma\|v_1\|^2_{H^1}, \quad \forall v_1\in V^1
\end{align}
where $c_\sigma$ (resp. $C_\sigma$) depends on $\mu_1$ (resp. $\mu_2$ and $\lambda_2$). Similarly, there are constants $c_\kappa$ (resp. $C_\kappa$) depending on the bound $\kappa_1$ (resp. $\kappa_2$) such that
\begin{align}\label{kappa-bounds}
c_\kappa\|v_2\|^2_{H^1}\leq (\kappa \nabla v_2, \nabla v_2) \leq C_\kappa\|v_2\|^2_{H^1}, \quad \forall v_2\in V^2.
\end{align}
Furthermore, we use the following notation for the energy norms induced by the bilinear forms 
\begin{align*}
\|v_1\|^2_\sigma := (\sigma(v_1):\varepsilon(v_1)), \ v_1\in V^1, \quad \|v_2\|^2_\kappa := (\kappa \nabla v_2 \nabla v_2), \ v_2\in V^2
\end{align*}

Existence and uniqueness of a solution to \eqref{weak1}-\eqref{weak2} have been proved in \cite{Showalter00,Zenisek84-2}. There are also some papers on the solution to contact problems, see \cite{Shillor93,Xu96}. 
\begin{mythm}	
	Assume that \ref{A1}-\ref{A3} hold and that $\partial \Omega$ is sufficiently smooth. Then there exist $u$ and $\theta$ such that $u\in L_2(V^1)$, $\nabla \cdot \dot u\in L_2(H^{-1})$, $\theta\in L_2(V^2)$, and $\dot \theta \in L_2(H^{-1})$ satisfying \eqref{weak1}-\eqref{weak2} and the initial condition $\theta(0,\cdot) = \theta_0$.
\end{mythm}

\begin{myremark}
	We remark that the equations \eqref{thermo1}-\eqref{initial} also describe a poroelastic system. In this case $\theta$ denotes the fluid pressure, $\kappa$ the permeability and viscosity of the fluid. 
\end{myremark}

\section{Numerical approximation}\label{sec:num}
In this section is we first recall some properties of the classical finite element method for \eqref{weak1}-\eqref{weak2}. In subsection~\ref{subsec:gfem} we propose a new numerical method built on the ideas from \cite{Malqvist14}. The localization of this method is treated in Section~\ref{sec:loc}.

\subsection{Classical finite element}
First, we need to define appropriate finite element spaces. For this purpose we let \{$\mathcal{T}_h\}_{h>0}$ be a family of shape regular triangulations of $\Omega$ with the mesh size $h_K:= \diam(K)$, for $K\in\mathcal{T}_h$. Furthermore, we denote the largest diameter in the triangulation by $h:=\max_{K\in \mathcal{T}_h} h_K$. We now define the classical piecewise affine finite element spaces
\begin{align*}
V^1_h &= \{v\in (C(\bar{\Omega}))^d: v=0 \text{ on } \Gamma^u_D, v|_K \text{ is a polynomial of degree} \leq 1, \forall K \in \mathcal{T}_h\}, \\
V^2_h &= \{v\in C(\bar{\Omega}): v=0 \text{ on } \Gamma^\theta_D, v|_K \text{ is a polynomial of degree} \leq 1, \forall K \in \mathcal{T}_h\}.
\end{align*}

For the discretization in time we consider, for simplicity, a uniform time step $\tau$ such that $t_n = n\tau$ for $n \in \{0,1,...,N\}$ and $N\tau=T$. 

\begin{myremark}
	The classical linear elasticity equation can in some cases suffer from locking effects when using continuous piecewise linear polynomials in both spaces (P1-P1 elements). These typically occur if $\nu$ is close to $1/2$ (Poisson locking) or if the thickness of the domain is very small (shear locking). In the coupled time-dependent problem locking can occur if $\dot \theta$ is neglected in \eqref{thermo2} and P1-P1 elements are used. The locking produces artificial oscillations in the numerical approximation of the temperature (or pressure) for early time steps. However, it shall be noted that in the case when $\dot \theta$ is \textit{not} neglected, this locking effect does not occur, see \cite{Phillips08}. Thus, we consider a P1-P1 discretization in this paper. 
\end{myremark}

The classical finite element method with a backward Euler scheme in time reads; for $n\in \{1,...,N\}$ find $u^n_h\in V^1_h$ and $\theta^n_h \in V^2_h$, such that
\begin{alignat}{2}
(\sigma(u^n_h):\varepsilon(v_1))- (\alpha \theta^n_h, \nabla \cdot v_1) &= (f^n,v_1),& \quad &\forall v_1 \in V^1_h,\label{fem1}\\
(\ddt\theta^n_h,v_2) + (\kappa \nabla \theta^n_h, \nabla v_2) + (\alpha \nabla \cdot \ddt u^n_h,v_2)&= (g^n,v_2),& &\forall v_2 \in V^2_h \label{fem2},
\end{alignat}
where $\ddt \theta^n_h:= (\theta^n_h-\theta^{n-1}_h)/\tau$ and similarly for $\ddt u^n_h$. The right hand sides are evaluated at time $t_n$, that is, $f^n := f(t_n)$ and $g^n := g(t_n)$. Given initial data $u^0_h$ and $\theta^0_h$ the system \eqref{fem1}-\eqref{fem2} is well posed \cite{Ern09}. We assume that $\theta^0_h\in V^1_h$ is a suitable approximation of $\theta_0$. For $u^0_h$ we note that $u(0)$ is uniquely determined by \eqref{weak1} at $t=0$, that is, $u(0)$ fulfills the equation
\begin{align*}
(\sigma(u(0)):\varepsilon(v_1)) - (\alpha \theta^0, \nabla\cdot v_1)=(f^0,v_1), \quad \forall v_1 \in V^1,
\end{align*} 
and we thus define $u^0_h \in V^1_h$ to be the solution to
\begin{align}\label{u0}
(\sigma(u^0_h):\varepsilon(v_1)) - (\alpha \theta^0_h, \nabla\cdot v_1)=(f^0,v_1), \quad \forall v_1 \in V^1_h.
\end{align}

The following theorem is a consequence of \cite[Theorem 3.1]{Ern09}. The convergence rate is optimal for the two first norms. However, it is not optimal for the $L_2$-norm $\|\theta^n-\theta^n_h\|$. In \cite{Ern09} this is avoided by using second order continuous piecewise polynomials for the displacement (P2-P1 elements). It is, however, noted that the problem is still stable using P1-P1 elements. In this paper we use P1-P1 elements and derive error bounds in the $L_\infty(H^1)$-norm, of optimal order, for both the displacement and the temperature. 
\begin{mythm}\label{femconv}
	Let $(u,\theta)$ be the solution to \eqref{weak1}-\eqref{weak2} and $\{(u^n_h,\theta^n_h)\}_{n=1}^N$ be the solution to \eqref{fem1}-\eqref{fem2}. Then for $n \in \{1,...,N\}$
	\begin{align*}
	\|u^n-u^n_h\|_{H^1} + \Big(\sum_{m=1}^n\tau\|\theta^m-\theta^m_h\|^2_{H^1}\Big)^{1/2} + \|\theta^n-\theta^n_h\| \leq C_{\epsilon^{-1}}(h+\tau),
	\end{align*}
	where $C_{\epsilon^{-1}}$ is of order $\epsilon^{-1}$ if the material varies on a scale of size $\epsilon$.
\end{mythm}
Note that the constant involved in this error bound contains derivatives of the coefficients. Hence, convergence only takes place when the mesh size $h$ is sufficiently small ($h<\epsilon$). Throughout this paper, it is assumed that $h$ is small enough and $V^1_h$ and $V^2_h$ are referred to as reference spaces for the solution. Similarly, $u^n_h$ and $\theta^n_h$ are referred to as reference solutions. In Section~\ref{sec:error} this solution is compared with the generalized finite element solution. We emphasize that the generalized finite element solution is computed in spaces of lower dimension and hence not as computationally expensive.

In the following theorem we prove some regularity results for the finite element solution.
\begin{mythm}\label{femreg}
	Let $\{u^n_h\}_{n=1}^N$ and $\{\theta^n_h\}_{n=1}^N$ be the solution to \eqref{fem1}-\eqref{fem2}. Then the following bound holds
	\begin{align}\label{femreg1}
	\Big(\sum_{j=1}^n\tau \|\ddt u^j_h\|^2_{H^1}\Big)^{1/2} &+\Big(\sum_{j=1}^n\tau \|\ddt \theta^j_h\|^2\Big)^{1/2} +  \|\theta^n_h\|_{H^1} \\&\leq C(\|g\|_{L_\infty(L_2)} + \|\dot f\|_{L_\infty(H^{-1})} + \|\theta^0_h\|_{H^1})\notag
	\end{align}
	If $\theta^0_h=0$, then for $n \in \{1,...,N\}$
	\begin{align}\label{femreg2}
	\|\ddt u^n_h\|_{H^1} &+ \|\ddt \theta^n_h\| +  \Big(\sum_{j=1}^n\tau\|\ddt \theta^j_h\|^2_{H^1}\Big)^{1/2}
	\\& \leq C\big(\|g\|_{L_\infty(L_2)} + \|\dot g \|_{L_\infty(H^{-1})} + \|\dot f\|_{L_\infty(H^{-1})} + \|\ddot f\|_{L_\infty(H^{-1})}\big).\notag
	\end{align}
	If $f=0$ and $g=0$, then for $n \in \{1,...,N\}$
	\begin{align}\label{femreg3}
	\|\ddt u^n_h\|_{H^1}  + \|\ddt \theta^n_h\| + t_n^{1/2}\|\ddt \theta^n_h\|_{H^1} \leq Ct_n^{-1/2}\|\theta^0_h\|_{H^1}.
	\end{align}
\end{mythm}
\begin{proof}
	From \eqref{fem1}-\eqref{fem2} and the initial data \eqref{u0} we deduce that the following relation must hold for $n \geq 1$
	\begin{alignat}{2}
	(\sigma(\ddt u^n_h):\varepsilon(v_1))- (\alpha \ddt\theta^n_h, \nabla \cdot v_1) &= (\ddt f^n,v_1),& \quad &\forall v_1 \in V^1_h,\label{udiff1}\\
	(\ddt\theta^n_h,v_2) + (\kappa \nabla \theta^n_h, \nabla v_2) + (\alpha \nabla \cdot \ddt u^n_h,v_2)&= (g^n,v_2),& &\forall v_2 \in V^2_h.\label{udiff2}
	\end{alignat}
	By choosing $v_1=\ddt u^n_h$ and $v_2=\ddt \theta^n_h$ and adding the resulting equations we have
	\begin{align}\label{eq:reg1}
	\frac{c_\sigma}{2}\|\ddt u^n_h\|^2_{H^1} + \frac{1}{2}\|\ddt \theta^n_h\|^2  + (\kappa \nabla \theta^n_h, \nabla \ddt \theta^n_h) \leq C(\|g^n\|^2 + \|\ddt f^n\|^2_{H^{-1}}).
	\end{align}
	Note that the coupling terms cancel. By using Cauchy-Schwarz and Young's inequality we can bound
	\begin{align*}
	\tau(\kappa \nabla \theta^n_h, \nabla \ddt \theta^n_h) = \|\kappa^{1/2}\nabla \theta^n_h\|^2 - (\kappa \nabla \theta^n_h, \nabla \theta^{n-1}_h) \geq \frac{1}{2}\|\theta^n_h\|^2_\kappa-\frac{1}{2}\|\theta^{n-1}_h\|^2_\kappa.
	\end{align*}
	Multiplying \eqref{eq:reg1} by $\tau$, summing over $n$, and using \eqref{sigma-bounds} gives
	\begin{align*}
	\sum_{j=1}^n\tau \|\ddt u^j_h\|^2_{H^1} +\sum_{j=1}^n\tau \|\ddt \theta^j_h\|^2 +  \|\theta^n_h\|^2_{H^1} &\leq C\sum_{j=1}^n\tau(\|g^j\|^2 + \|\ddt f^j\|^2_{H^{-1}}) \\&\quad+ C\|\theta^0_h\|_{H^1},
	\end{align*}
	which is bounded by the right hand side in \eqref{femreg1}. 
	
	For the bound \eqref{femreg2} we note that the following relation must hold for $n \geq 2$
	\begin{alignat}{2}
	(\sigma(\ddt u^n_h):\varepsilon(v_1))- (\alpha \ddt\theta^n_h, \nabla \cdot v_1) &= (\ddt f^n,v_1),& \quad &\forall v_1 \in V^1_h,\label{femdiff1}\\
	(\ddt^2\theta^n_h,v_2) + (\kappa \nabla \ddt \theta^n_h, \nabla v_2) + (\alpha \nabla \cdot \ddt^2 u^n_h,v_2)&= (\ddt g^n,v_2),& &\forall v_2 \in V^2_h.\label{femdiff2}
	\end{alignat}
	Now choose $v_1=\ddt^2u^n_h$ and $v_2=\ddt \theta^n_h$ and add the resulting equations to get 
	\begin{align*}
	(\sigma(\ddt u^n_h):\varepsilon(\ddt^2u^n_h)) + (\ddt^2\theta^n_h,\ddt\theta^n_h) &+ (\kappa \nabla \ddt \theta^n_h, \nabla \ddt \theta^n_h) \\
	&= (\ddt f^n,\ddt^2u^n_h) + (\ddt g^n,\ddt\theta^n_h).
	\end{align*}
	Multiplying by $\tau$ and using Cauchy-Schwarz and Young's inequality gives
	\begin{align*}
	\frac{1}{2}\|\ddt u^n_h\|^2_\sigma + \frac{1}{2}\|\ddt\theta^n_h\|^2 + C\tau \|\ddt \theta^n_h\|^2_{H^1} &\leq \frac{1}{2}\|\ddt\theta^{n-1}_h\|^2 + \frac{1}{2}\|\ddt u^{n-1}_h\|^2_\sigma  \\&\quad+ \tau(\ddt f^n,\ddt^2u^n_h)  + C\|\ddt g^n\|^2_{H^{-1}}.
	\end{align*}
	Summing over $n$ and using \eqref{sigma-bounds} now gives
	\begin{align*}
	\|\ddt u^n_h\|^2_{H^1} +\|\ddt\theta^n_h\|^2 +  \sum_{j=2}^n\tau \|\ddt \theta^j_h\|^2_{H^1}  &\leq C\Big(\|\ddt u^1_h\|^2_{H^1} + \|\ddt\theta^1_h\|^2 \\&\quad + \sum_{j=2}^n\tau\big((\ddt f^j,\ddt^2u^j_h) + \|\ddt g^j\|^2_{H^{-1}}\big)\Big).
	\end{align*}
	Here we use summation by parts to get
	\begin{align*}
	\sum_{j=2}^n\tau(\ddt f^j,\ddt^2u^j_h) &= (\ddt f^n,\ddt u^n_h)-(\ddt f^1,\ddt u^1_h) - \sum_{j=2}^n\tau(\ddt^2 f^j,\ddt u^{j-1}_h)\\
	&\leq C\bigg(\max_{1\leq j \leq n}\|\ddt f^j\|_{H^{-1}} + \sum_{j=2}^n\tau\|\ddt^2 f^j\|_{H^{-1}}\bigg)\max_{1\leq j \leq n}\|\ddt u^j_h\|_{H^1}, 
	\end{align*}
	and $\max_{1\leq j \leq n}\|\ddt u^j_h\|_{H^1}$ can now be kicked to the left hand side. 
	
	To estimate $\ddt\theta^1_h$ and $\ddt u^1_h$ we choose $v_1=\ddt u^1_h$ and $v_2=\ddt \theta^1_h$ in \eqref{udiff1}-\eqref{udiff2} for $n=1$. We thus have, since $\theta^0_h=0$,
	\begin{align*}
	\|\ddt u^1_h\|^2_{H^1} + \|\ddt\theta^1_h\|^2 + \frac{1}{\tau}\|\theta^1_h\|^2_{H^1} \leq C(\|\ddt f^1\|^2_{H^{-1}} + \|g^1\|^2).
	\end{align*}
	The observation that $\frac{1}{\tau}\|\theta^1_h\|^2_{H^1}=\tau\|\ddt\theta^1_h\|^2_{H^1}$ completes the bound \eqref{femreg2}.
	
	Now assume $f=0$ and $g=0$ and note that the following holds for $n\geq 2$, 
	\begin{alignat*}{2}
	(\sigma(\ddt^2 u^n_h):\varepsilon(v_1))- (\alpha \ddt^2\theta^n_h, \nabla \cdot v_1) &= 0,& \quad &\forall v_1 \in V^1_h,\\
	(\ddt^2\theta^n_h,v_2) + (\kappa \nabla \ddt \theta^n_h, \nabla v_2) + (\alpha \nabla \cdot \ddt^2 u^n_h,v_2)&=0,& &\forall v_2 \in V^2_h.
	\end{alignat*}
	Choosing $v_1=\ddt^2 u^n_h$, $v_2=\ddt^2\theta^n_h$ and adding the resulting equations gives 
	\begin{align*}
	(\sigma(\ddt^2 u^n_h):\varepsilon(\ddt^2u^n_h)) + (\ddt^2\theta^n_h,\ddt^2\theta^n_h) + (\kappa \nabla \ddt \theta^n_h, \nabla \ddt^2\theta^n_h) = 0,
	\end{align*}
	where, again, the coupling terms cancel. The two first terms on the left hand side are positive and can thus be ignored. Multiplying by $\tau$ and $t_n^2$ gives after using Cauchy-Schwarz and Young's inequality
	\begin{align*}
	t_n^2\|\ddt \theta^n_h\|^2_\kappa \leq t_{n-1}^2\|\ddt\theta^{n-1}_h\|^2_\kappa + (t_n^2-t_{n-1}^2)\|\ddt^2\theta^{n-1}_h\|^2_\kappa.
	\end{align*}
	Note that $t_n^2-t_{n-1}^2\leq 3\tau t_{n-1}$, where we use that $t_n\leq 2t_{n-1}$ if $n\geq 2$. Summing over $n$ now gives
	\begin{align*}
	t_n^2\|\ddt \theta^n_h\|^2_\kappa \leq t_1^2\|\ddt\theta^1_h\|^2_\kappa + 3\sum_{j=2}^n\tau t_{j-1}\|\ddt\theta^{j-1}_h\|^2_\kappa.
	\end{align*}
	To bound the last sum we choose $v_1=\ddt^2 u^n_h$, $v_2=\ddt\theta^n_h$ in \eqref{femdiff1}-\eqref{femdiff2}, now with $f=0$ and $g=0$. Adding the resulting equations gives
	\begin{align*}
	(\ddt^2\theta^n_h,\ddt\theta^n_h) + (\kappa \nabla \ddt \theta^n_h, \nabla \ddt \theta^n_h) + (\sigma(\ddt u^n_h):\varepsilon(\ddt^2u^n_h))= 0,
	\end{align*}
	Multiplying by $\tau$ and $t_n$ gives after using Cauchy-Schwarz inequality
	\begin{align*}
	\frac{t_n}{2}&\|\ddt u^n_h\|^2_\sigma + \frac{t_n}{2}\|\ddt\theta^n_h\|^2 + c_\kappa\tau t_n\|\ddt \theta^n_h\|^2_{H^1} \\&\leq \frac{t_{n-1}}{2}\|\ddt u^{n-1}_h\|^2_\sigma  +\frac{t_{n-1}}{2}\|\ddt\theta^{n-1}_h\|^2 + \frac{\tau}{2}\|\ddt u^{n-1}_h\|^2_\sigma+  \frac{\tau}{2}\|\ddt\theta^{n-1}_h\|^2 .
	\end{align*}
	Summing over $n$ and using \eqref{sigma-bounds} thus gives
	\begin{align*}
	&\frac{c_\sigma t_n}{2}\|\ddt u^n_h\|^2_{H^1} + \frac{t_n}{2}\|\ddt\theta^n_h\|^2 + \sum_{j=2}^n\tau t_j\|\ddt \theta^j_h\|^2_{H^1}  \\
	&\quad \leq \frac{C_\sigma t_1}{2}\|\ddt u^1_h\|^2_{H^1}  + \frac{t_1}{2}\|\ddt\theta^1_h\|^2 + C\sum_{j=2}^n\tau\big(\|\ddt u^{j-1}_h\|^2_{H^1} + \|\ddt\theta^{j-1}_h\|^2\big).
	\end{align*}
	To bound the last sum in this estimate we choose $v_1=\ddt u^n_h$, $v_2=\ddt\theta^n_h$ in \eqref{udiff1}-\eqref{udiff2} and multiply by $\tau$ to get
	\begin{align*}
	 c_\sigma\tau\|\ddt u^n_h\|^2_{H^1} + \tau\|\ddt\theta^n_h\|^2 + \frac{1}{2}\|\theta^n_h\|^2_\kappa \leq  \frac{1}{2}\|\theta^{n-1}_h\|^2_\kappa.
	\end{align*} 
	Summing over $n$ and using \eqref{kappa-bounds} gives
	\begin{align}\label{eq:sumtheta}
	C\sum_{j=1}^n\tau\big(\|\ddt\theta^j_h\|^2 + \|\ddt u^j_h\|^2_{H^1}\big) + \frac{c_\kappa}{2}\|\theta^n_h\|^2_{H^1}\leq\frac{C_\kappa}{2}\|\theta^0_h\|^2_{H^1}.
	\end{align} 
	It remains to bound $t^2_1\|\ddt\theta^1_h\|^2_{H^1}$, $t_1\|\ddt \theta^1_h\|^2$, and $t_1\|\ddt u^1_h\|_{H^1}$. For this purpose we recall that $t_1=\tau$ and use \eqref{eq:sumtheta} for $n=1$ to get 
	\begin{align*}
	t_1\|\ddt u^1_h\|_{H^1} &+ t_1\|\ddt \theta^1_h\|^2 + t^2_1\|\ddt\theta^1_h\|^2_{H^1} \\&\leq C(\tau(\|\ddt u^1_h\|^2_{H^1} + \|\ddt \theta^1_h\|^2) + \|\theta^1_h\|^2_{H^1} + \|\theta^0_h\|^2_{H^1}) \leq C\|\theta^0\|^2_{H^1}.
	\end{align*}
	Finally, we have that
	\begin{align*}
	t_n\|\ddt u^n_h\|^2_{H^1} + t_n\|\ddt \theta^n_h\|^2 \leq C\|\theta^0\|^2_{H^1},\quad t_n^2\|\ddt \theta^n_h\|^2_{H^1} \leq C\|\theta^0\|^2_{H^1},
	\end{align*}
	and thus \eqref{femreg3} follows.
\end{proof}



\subsection{Generalized finite element}\label{subsec:gfem}
In this section we shall derive a generalized finite element method. First we define $V^1_H$ and $V^2_H$ analogously to $V^1_h$ and $V^2_h$, but with a larger mesh size $H>h$. In addition, we assume that the family of triangulations $\{\mathcal T_H\}_{H>h}$ is quasi-uniform and that $\mathcal{T}_h$ is a refinement of $\mathcal{T}_H$ such that $V^1_H \subseteq V^1_h$ and $V^2_H \subseteq V^2_h$. Furthermore, we use the notation $\mathcal N = \mathcal N^1 \times \mathcal N^2$ to denote the free nodes in $V^1_H \times V^2_H$. The aim is now to define a new (multiscale) space with the same dimension as $V^1_H\times V^2_H$, but with better approximation properties. 
For this purpose we define an interpolation operator $I_H=(I^1_H,I^2_H):V^1_h\times V^2_h\rightarrow V^1_H\times V^2_H$ with the property that $I_H \circ I_H = I_H$ and for all $v=(v_1,v_2) \in V^1_h \times V^2_h$
\begin{align}\label{interpolation}
H^{-1}_K\|v-I_Hv\|_{L_2(K)} + \|\nabla I_Hv\|_{L_2(K)} \leq C_I\|\nabla v\|_{L_2(\omega_K)}, \quad \forall K\in \mathcal{T}_H,
\end{align}
where
\begin{align*}
\omega_K:=\text{int } \{\hat K \in \mathcal T_H: \hat K \cap K\neq \emptyset\}.
\end{align*}
Since the mesh is assumed to be shape regular, the estimates in \eqref{interpolation} are also global, i.e.,
\begin{align}\label{interpolation-global}
H^{-1}\|v-I_Hv\|+ \|\nabla I_Hv\| \leq C\|\nabla v\|,
\end{align}
where $C$ is a constant depending on the shape regularity parameter,  $\gamma>0$;
\begin{align}\label{shape}
\gamma:=\max_{K \in \mathcal{T}_H} \gamma_K, \ \text{with} \ \gamma_K:= \frac{\diam B_K}{\diam K}, \ \text{for}\ K\in \mathcal{T}_H,
\end{align}
where $B_K$ is the largest ball contained in $K$.


One example of an interpolation that satisfies the above assumptions is $I^i_H=E^i_H\circ \Pi^i_H$, $i=1,2$. Here $\Pi^i_H$ denotes the piecewise $L_2$-projection onto $P_1(\mathcal{T}_H)$ ($P_1(\mathcal{T}_H)^d$ if $i=1$), the space of functions that are affine on each triangle $K\in \mathcal{T}_H$. Furthermore, $E^1_H$ is an averaging operator mapping $(P_1(\mathcal{T}_H))^d$ into $V^1_H$, by (coordinate wise)
\begin{align*}
(E^{1,j}_H(v))(z) = \frac{1}{\card\{K\in \mathcal{T}_H: z \in K\}}\sum_{K\in \mathcal{T}_H: z \in K} v^j|_K(z), \quad 1\leq j\leq d,
\end{align*} 
where $z\in \mathcal N^1$. $E^2_H$ mapping $\mathcal P^1_H$ to $V^2_H$ is defined similarly. For a further discussion on this interpolation and other available options we refer to \cite{Peterseim15}. 

Let us now define the kernels of $I^1_H$ and $I^2_H$ 
\begin{align*}
V^1_\f :=\{v \in V^1_h: I^1_Hv=0\},\quad V^2_\f :=\{v \in V^2_h: I^2_Hv=0\}
\end{align*}
The kernels are fine scale spaces in the sense that they contain all features that are not captured by the (coarse) finite element spaces $V^1_H$ and $V^2_H$. Note that the interpolation leads to the splits $V^1_h=V^1_H \oplus V^1_\f $ and $V^2_h=V^2_H \oplus V^2_\f $, meaning that any function $v_1 \in V^1_h$ can be uniquely decomposed as $v_1 = v_{1,H} + v_{1,\f}$, with $v_{1,H} \in V^1_H$ and $v_{1,\f}\in V^1_\f$, and similarly for $v_2 \in V^2_h$.

Now, we introduce a Ritz projection onto the fine scale spaces. For this we use the bilinear forms associated with the diffusion in \eqref{weak1}-\eqref{weak2}. The projection of interest is thus $R_\f : V^1_h\times V^2_h \rightarrow V^1_\f \times V^2_\f $, such that for all $(v_1,v_2) \in V^1_h \times V^2_h$, $R_\f (v_1,v_2)=(R^1_\f v_1,R^2_\f v_2)$ fulfills
\begin{alignat}{2}
(\sigma(v_1 - R^1_\f v_1):\varepsilon(w_1))&= 0,& \quad &\forall w_1 \in V^1_\f ,\label{fritz1}\\
(\kappa \nabla (v_2-R^2_\f v_2), \nabla w_2) &= 0,& &\forall w_2 \in V^2_\f  \label{fritz2}.
\end{alignat}
Note that this is an uncoupled system and $R^1_\f $ and $R^2_\f $ are classical Ritz projections.

For any $(v_1,v_2)\in V^1_h\times V^2_h$ we have, due to the splits of the spaces $V^1_h$ and $V^2_h$ above, that
\begin{align*}
v_1 - R^1_\f v_1 = v_{1,H}-R^1_\f v_{1,H}, \quad v_2 - R^2_\f v_2 = v_{2,H}-R^2_\f v_{2,H}.
\end{align*}
Using this we define the multiscale spaces 
\begin{align}\label{msspace}
V^1_\ms :=\{v-R^1_\f v:v\in V^1_H\}, \quad V^2_\ms :=\{v-R^2_\f v:v\in V^2_H\}.
\end{align}
Clearly $V^1_\ms \times V^2_\ms $ has the same dimension as $V^1_H \times V^2_H$. Indeed, with $\lambda^1_x$ denoting the hat function in $V^1_H$ at node $x$ and $\lambda^2_y$ the hat function in $V^2_H$ at node $y$, such that
\begin{align*}
V^1_H\times V^2_H = \Span \{(\lambda^1_x,0),(0,\lambda^2_y):(x,y)\in \mathcal N\},
\end{align*}
a basis for $V^1_\ms \times V^2_\ms $ is given by 
\begin{align}\label{msbasis}
\{(\lambda^1_x-R^1_\f \lambda^1_x,0),(0,\lambda^2_y-R^2_\f \lambda^2_y):(x,y)\in \mathcal N\}.
\end{align}

Finally, we also note that the splits $V^1_h=V^1_\ms  \oplus V^1_\f $ and $V^2_h=V^2_\ms  \oplus V^2_\f $ hold, which fulfill the following orthogonality relation
\begin{alignat}{2}\label{orthog}
(\sigma(v_1):\varepsilon(w_1)&= 0,& \quad &\forall v_1\in V^1_\ms,\, w_1\in V^1_\f ,\\
(\kappa \nabla v_2, \nabla w_2) &= 0,& &\forall v_2\in V^2_\ms,\, w_2 \in V^2_\f 
\end{alignat}

\subsubsection{Stationary problem}
For the error analysis in Section~\ref{sec:error} it is convenient to define the Ritz projection onto the multiscale space using the bilinear form given by the stationary version of \eqref{weak1}-\eqref{weak2}. We thus define $R_\ms : V^1_h\times V^2_h \rightarrow V^1_\ms \times V^2_\ms $, such that for all $(v_1,v_2) \in V^1_h \times V^2_h$, $R_\ms (v_1,v_2)=(R^1_\ms (v_1,v_2),R^2_\ms v_2)$ fulfills
\begin{alignat}{2}
(\sigma(v_1 - R^1_\ms (v_1,v_2)):\varepsilon(w_1)) - (\alpha(v_2-R^2_\ms v_2),\nabla\cdot w_1)&=0,& \quad &\forall w_1 \in V^1_\ms , \label{msritz1}\\
(\kappa \nabla (v_2-R^2_\ms v_2), \nabla w_2&)=0, & &\forall w_2\in V^2_\ms. \label{msritz2}
\end{alignat}
Note that we must have $R^2_\ms = I-R^2_\f $, but $R^1_\ms \neq I - R^1_\f$ in general. 

The Ritz projection in \eqref{msritz1}-\eqref{msritz2} is upper triangular. Hence, when solving for $R^1_\ms(v_1,v_2)$ the term $(\alpha R^2_\ms v_2 ,\nabla\cdot w_1)$ in \eqref{msritz2} is known. Since this term has multiscale features and appears on the right hand side, we impose a correction on $R^1_\ms(v_1,v_2)$ inspired by the ideas in \cite{Henning14} and \cite{Larson07}. The correction is defined as the element $\tilde R_\f v_2  \in V^1_\f $, which fulfills  
\begin{align}\label{stationarycorr}
(\sigma(\tilde R_\f v_2):\varepsilon(w_1)) = (\alpha R^2_\ms v_2 ,\nabla\cdot w_1),\quad \forall w_1 \in V^1_\f,
\end{align}
and we define $\tilde R^1_\ms(v_1,v_2) = R^1_\ms(v_1,v_2) + \tilde R_\f v_2$.  

Note that the Ritz projections are stable in the sense that
\begin{align}\label{msstable}
\|\tilde R^1_{\ms}(v_1,v_2)\|_{H^1}\leq C(\|v_1\|_{H^1} + \|v_2\|_{H^1}), \quad \|R^2_{\ms}v_2\|_{H^1}\leq C\|v_2\|_{H^1}.
\end{align}

\begin{myremark}
	The problem to find $\tilde R_\f v_2$ is posed in the entire fine scale space and is thus computationally expensive to solve. The aim is to localize these computations to smaller patches of coarse elements, see Section~\ref{sec:loc}.
\end{myremark}

To derive error bounds for this projection we define two operators $\mathcal A_1:V^1_h\times V^2_h\rightarrow V^1_h$ and $\mathcal A_2:V^2_h \rightarrow V^2_h$ such that for all $(v_1,v_2)\in V^1_h\times V^2_h$ we have
\begin{align}
(\mathcal A_1(v_1,v_2),w_1) &= (\sigma(v_1):\varepsilon(w_1)) - (\alpha v_2, \nabla \cdot w_1), \quad \forall w_1\in V^1_h,\label{A1operator} \\
(\mathcal A_2v_2,w_2) &= (\kappa \nabla v_2, \nabla w_2), \quad \forall w_2\in V^2_h.\label{A2operator}
\end{align}

\begin{mylemma}\label{stationaryconv}
	For all $(v_1,v_2)\in V^1_h\times V^2_h$ it holds that
	\begin{align}
	\label{stationaryconv1}\|v_1 - \tilde R^1_\ms(v_1,v_2) \|_{H^1} &\leq C(H\|\mathcal A_1(v_1,v_2)\| + \|v_2 - R^2_\ms v_2\|)\\
	&\leq CH(\|\mathcal A_1(v_1,v_2)\|+ \|v_2\|_{H^1}), \notag\\
	\|v_2 - R^2_\ms v_2\|_{H^1} &\leq CH\|\mathcal A_2 v_2\|. \label{stationaryconv2}
	\end{align}
\end{mylemma}
\begin{proof}
	It follows from \cite{Malqvist14} that \eqref{stationaryconv2} holds, since \eqref{msritz2} is an elliptic equation of Poisson type. Using an Aubin-Nitsche duality argument as in, e.g., \cite{MP15}, we can derive the following estimate in the $L_2$-norm
	\begin{align*}
	\|v_2 - R^2_\ms v_2 \|\leq CH\|v_2 - R^2_\ms v_2\|_{H^1} \leq CH\|v_2\|_{H^1},
	\end{align*}
	which proves the second inequality in \eqref{stationaryconv1}. 
	
	It remains to bound $\|v_1-\tilde R^1_\ms(v_1,v_2)\|_{H^1}$. Recall that any $v\in V^1_h$ can be decomposed as 
	\begin{align*}
	v=v-R^1_\f v + R^1_\f v =(I-R^1_\f)v + R^1_\f v,
	\end{align*} 
	where $(I-R^1_\f)v \in V^1_\ms$. Using the orthogonality \eqref{orthog} and that $(\sigma(\cdot):\varepsilon(\cdot))$ is a symmetric bilinear form we get
	\begin{align*}
	(\sigma(\tilde R^1_\ms(v_1,v_2)):\varepsilon(v)) &= (\sigma(R^1_\ms(v_1,v_2) + \tilde R_\f v_2):\varepsilon((I-R^1_\f)v + R^1_\f v)) \\&= (\sigma(R^1_\ms(v_1,v_2)):\varepsilon((I-R^1_\f) v)) + (\sigma(\tilde R_\f v_2):\varepsilon(R^1_\f v)).
	\end{align*}
	Due to \eqref{msritz1} and \eqref{stationarycorr} we thus have
	\begin{align*}
	&(\sigma(R^1_\ms(v_1,v_2)):\varepsilon((I-R^1_\f) v)) + (\sigma(\tilde R_\f v_2):\varepsilon(R^1_\f v))\\
	&\,=(\sigma(v_1):\varepsilon((I-R^1_\f)v)) - (\alpha (v_2-R^2_\ms v_2),\nabla \cdot (I-R^1_\f)v) + (\alpha R^2_\ms v_2,\nabla \cdot R^1_\f v) \\
	&\,=(\mathcal A_1(v_1,v_2),(I-R^1_\f)v) + (\alpha R^2_\ms v_2,\nabla \cdot v).
	\end{align*}
	Define $e:=v_1-\tilde R^1_\ms(v_1,v_2)$. Using the above relation together with \eqref{A1operator} we get the bound
	\begin{align*}
	c_\sigma\|e\|^2_{H^1}&\leq (\sigma(e):\varepsilon(e)) = (\sigma(v_1):\varepsilon(e)) - (\mathcal A_1(v_1,v_2),(I-R^1_\f)e) - (\alpha R^2_\ms v_2, \nabla \cdot e) \\
	&=(\mathcal A_1(v_1,v_2),R^1_\f e) + (\alpha(v_2-R^2_\ms v_2),\nabla\cdot e)\\
	&\leq\|\mathcal{A}_1(v_1,v_2)\|\|R^1_\f e\| + C\|v_2-R^2_\ms v_2\|\|e\|_{H^1}
	\end{align*}
	Since $R^1_\f e\in V^1_\f$ we have due to \eqref{interpolation}
	\begin{align*}
	\|R^1_\f e\|&=\|R^1_\f e-I^1_HR^1_\f e\|\leq CH\|R^1_\f e\|_{H^1} \leq CH\|e\|_{H^1},
	\end{align*}
	where we have used the stability $\|R^1_\f v\|_{H^1}\leq C\|v\|_{H^1}$ for $v\in V^1_h$. 
	The first inequality in \eqref{stationaryconv1} now follows.
\end{proof}

\begin{myremark}
	Without the correction $\tilde R_\f$ the error bound \eqref{stationaryconv1} would depend on the derivatives of $\alpha$,
	\begin{align*}
	\|v_1 - R^1_\ms(v_1,v_2) \|_{H^1} &\leq C_{\alpha'}(H\|\mathcal A_1(v_1,v_2)\| + \|v_2 - R^2_\ms v_2\|),
	\end{align*}
	where $\alpha'$ is large if $\alpha$ has multiscale features.
\end{myremark}

\subsubsection{Time-dependent problem}
A generalized finite element method with a backward Euler discretization in time is now defined by replacing $V^1_h$ with $V^1_\ms$ and $V^2_h$ with $V^2_\ms$ in \eqref{fem1}-\eqref{fem2} and adding a correction similar to \eqref{stationarycorr}. The method thus reads; for $n\in \{1,...,N\}$ find $\tilde u^n_\ms=u^n_\ms + u^n_\f$, with $u^n_\ms\in V^1_\ms$, $u^n_\f \in V^1_\f$, and $\theta^n_\ms \in V^2_\ms$, such that
\begin{alignat}{2}
(\sigma(\tilde u^n_\ms):\varepsilon(v_1))- (\alpha \theta^n_\ms, \nabla \cdot v_1) &= (f^n,v_1),& \quad &\forall v_1 \in V^1_\ms,\label{gfem1}\\
(\ddt\theta^n_\ms,v_2) + (\kappa \nabla \theta^n_\ms, \nabla v_2) + (\alpha \nabla \cdot \ddt \tilde u^n_\ms,v_2)&= (g^n,v_2),& &\forall v_2 \in V^2_\ms \label{gfem2},\\
(\sigma(u^n_\f):\varepsilon(w_1)) - (\alpha \theta^n_\ms, \nabla\cdot w_1)&=0,& &\forall w_1 \in V^1_\f.\label{timecorr}
\end{alignat}
where $\theta^0_\ms=R^2_\ms \theta^0_h$. Furthermore, we define $\tilde u^0_\ms:=u^0_\ms + u^0_\f$, where $u^0_\f\in V^1_\f$ is defined by \eqref{timecorr} for $n=0$ and $u^0_\ms\in V^1_\ms$, such that
\begin{align}\label{u0ms}
(\sigma(\tilde u^0_\ms):\varepsilon(v_1)) - (\alpha \theta^0_\ms, \nabla\cdot v_1)=(f^0,v_1), \quad \forall v_1 \in V^1_\ms.
\end{align}
\begin{mylemma}\label{wellposed}
	The problem \eqref{gfem1}-\eqref{gfem2} is well-posed.
\end{mylemma}
\begin{proof}
	Given $u^{n-1}_\ms$, $\theta^{n-1}_\ms$, and $u^{n-1}_\f$, the equations \eqref{gfem1}-\eqref{timecorr} yields a square system. Hence, it is sufficient to prove that the solution is unique. Let $v_1=u^n_{\ms}-u^{n-1}_{\ms}$ in \eqref{gfem1} and $v_2=\tau\theta^n_{\ms}$ in \eqref{gfem2} and add the resulting equations to get
	\begin{align*}
		(\sigma(u^n_{\ms}):\varepsilon(u^n_{\ms}-u^{n-1}_{\ms})) &+ (\sigma(u^n_{\f}):\varepsilon(u^n_{\ms}-u^{n-1}_{\ms}))+ \tau(\ddt \theta^n_{\ms},\theta^n_{\ms}) \\ &+ c_\kappa\tau\|\theta^n_{\ms}\|^2_{H^1} + (\alpha \nabla\cdot (u^n_{\f}-u^{n-1}_{\f}), \theta^n_{\ms,k})
		\\&\leq  (f^n,u^n_{\ms}-u^{n-1}_{\ms}) + \tau(g^n,\theta^n_{\ms}).
	\end{align*}
	Using the orthogonality \eqref{orthog} and \eqref{timecorr} this simplifies to
	\begin{align*}
	(\sigma(u^n_{\ms}):\varepsilon(u^n_{\ms}-u^{n-1}_{\ms})) &+ \tau(\ddt \theta^n_{\ms},\theta^n_{\ms}) + c_\kappa\tau\|\theta^n_{\ms}\|^2_{H^1} + c_\sigma\|u^n_{\f}\|^2_{H^1}
	\\&\leq  (f^n,u^n_{\ms}-u^{n-1}_{\ms}) + \tau(g^n,\theta^n_{\ms}) +  (\sigma(u^n_\f):\varepsilon(u^{n-1}_{\f})).
	\end{align*}
	Now, using that $(\sigma(\cdot):\varepsilon(\cdot))$ is a symmetric bilinear form we get the following identity 
	\begin{align}\label{symmetry}
		(\sigma(v):\varepsilon(v-w)) = \frac{1}{2}(\sigma(v):\varepsilon(v)) &+ \frac{1}{2}(\sigma(v-w):\varepsilon(v-w))\\& - \frac{1}{2}(\sigma(w):\varepsilon(w)),\notag
	\end{align}
	and using Cauchy-Schwarz and Young's inequality we derive
	\begin{align*}
		(f^n,u^n_{\ms}-u^{n-1}_{\ms}) \leq C\|f^n\|_{H^{-1}} + \frac{1}{2}(\sigma(u^n_{\ms}-u^{n-1}_{\ms}):\varepsilon(u^n_{\ms}-u^{n-1}_{\ms})).
	\end{align*}
	This, together with the estimate $\tau(\ddt \theta^n_{\ms},\theta^n_{\ms})\geq \frac{1}{2}\|\theta^n_{\ms}\|^2-\frac{1}{2}\|\theta^{n-1}_{\ms}\|^2$ and \eqref{sigma-bounds}, leads to
	\begin{align*}
		\frac{c_\sigma}{2}\|u^n_{\ms}\|^2_{H^1} &+ \frac{1}{2}\|\theta^n_{\ms}\|^2 +  c_\kappa\tau\|\theta^n_{\ms}\|^2_{H^1} +
		\frac{c_\sigma}{2}\|u^n_{\f}\|^2_{H^1} \\&\leq  C(\|f^n\|^2_{H^{-1}} + \tau\|g^n\|^2 + \|\theta^{n-1}_{\ms}\|^2 + \|u^{n-1}_{\ms}\|^2_{H^1} + \|u^{n-1}_{\f}\|^2_{H^1}).
	\end{align*}
	Hence, a unique solution exists.
\end{proof}

\section{Localization}\label{sec:loc}
In this section we show how to truncate the basis functions, which is motivated by the exponential decay of \eqref{msbasis}. We consider a localization inspired by the one proposed in \cite{Henning14}, which is performed by restricting the fine scale space to patches of coarse elements defined by the following; for $K\in \mathcal T_H$ 
\begin{align*}
\omega_{0}(K)&:=\text{int } K,\\ \omega_{k}(K)&:=\text{int }\big(\cup \{\hat K \in \mathcal T_H: \hat K \cap \overline{\omega_{k-1}(K)}\neq \emptyset\}\big), \quad k=1,2,...
\end{align*}
Now let $V^1_{\f }(\omega_k(K)):=\{v\in V^1_\f :v(z)=0 \text{ on }(\overline{\Omega}\setminus\Gamma^u_N)\setminus \omega_k(K)\}$ be the restriction of $V^1_\f $ to the patch $\omega_k(T)$. We define $V^2_\f (\omega_k(K))$ similarly.

The localized fine scale space can now be used to approximate the fine scale part of the basis functions in \eqref{msbasis}, which significantly reduces the computational cost for these problems. Let $(\cdot,\cdot)_{\omega}$ denote the $L_2$ inner product over a subdomain $\omega \subseteq \Omega$ and define the local Ritz projection $R^K_{\f ,k}:V^1_h\times V^2_h\rightarrow V^1_\f (\omega_k(K))\times V^2_\f (\omega_k(K))$ such that for all $(v_1,v_2)\in V^1_h\times V^2_h$, $R^K_{\f ,k}(v_1,v_2)=(R^{K,1}_{\f , k}v_1,R^{K,2}_{ \f ,k}v_1)$ fulfills 
\begin{alignat}{2}
(\sigma(R^{K,1}_{\f ,k}v_1):\varepsilon(w_1))_{\omega_k(K)}&= (\sigma(v_1):\varepsilon(w_1))_{K},& \quad &\forall  w_1 \in V^1_\f (\omega_k(K)),\label{Kfritz1}\\
(\kappa \nabla (R^{K,2}_{\f ,k}v_2), \nabla w_2)_{\omega_k(K)} &= (\kappa \nabla v_2, \nabla w_2)_K,&  &\forall w_2 \in V^2_\f (\omega_k(K)) \label{Kfritz2}.
\end{alignat} 
Note that if we replace $\omega_k(K)$ with $\Omega$ in \eqref{Kfritz1}-\eqref{Kfritz2} and denote the resulting projection $R^K_\f(v_1,v_2) =(R^{K,1}_\f v_1 ,R^{K,2}_\f v_2)$, then for all $(v_1,v_2)\in V^1_h\times V^2_h$ we have
\begin{align*}
R_\f (v_1,v_2)=\sum_{K\in \mathcal T_H}R^K_{\f }(v_1,v_2)=\sum_{K\in \mathcal T_H}(R^{K,1}_\f v_1,R^{K,2}_\f v_2).
\end{align*}
Motivated by this we now define the localized fine scale projection as
\begin{align}\label{loc-fritz}
R_{\f ,k}(v_1,v_2):=\sum_{K\in \mathcal T_H}R^K_{\f ,k}(v_1,v_2)=\sum_{K\in \mathcal T_H}(R^{K,1}_{\f ,k}v_1,R^{K,2}_{\f ,k}v_2),
\end{align}
and the localized multiscale spaces
\begin{align}\label{mskspace}
V^1_{\ms ,k}:=\{v_1-R^1_{\f ,k}v_1: v_1\in V^1_H\}, \quad V^2_{\ms ,k}:=\{v_2-R^2_{\f ,k}v_2: v_2\in V^2_H\},
\end{align}
with the corresponding localized basis
\begin{align}\label{mskbasis}
\{(\lambda^1_x-R^1_{\f ,k}\lambda_x,0),(0,\lambda^2_y-R^2_{\f ,k}\lambda_y):(x,y)\in \mathcal N\}.
\end{align}

\subsection{Stationary problem}
In this section we define a localized version of the stationary problem \eqref{msritz1}-\eqref{msritz2}. Let $R_{\ms,k} : V^1_h\times V^2_h \rightarrow V^1_{\ms,k} \times V^2_{\ms,k}$, such that for all $(v_1,v_2) \in V^1_h \times V^2_h$, $R_{\ms,k} (v_1,v_2)=(R^1_{\ms,k} (v_1,v_2),R^2_{\ms,k} v_2)$. The method now reads; find 
\begin{align*}
\tilde R^1_{\ms,k} (v_1,v_2) = R^1_{\ms,k} (v_1,v_2) + \sum_{K\in \mathcal T_H} \tilde R^K_{\f,k}v_2, \quad \text{where } \tilde R^K_{\f,k}v_2 \in V^1_\f(\omega_k(K)),
\end{align*}
and $R^2_{\ms,k}v_2$ such that
\begin{alignat}{2}
\begin{aligned}[b]
(\sigma(v_1 - \tilde R^1_{\ms,k} &(v_1,v_2)):\varepsilon(w_1))  \label{mskritz1}  \\
&- (\alpha(v_2-R^2_{\ms,k} v_2),\nabla\cdot w_1)\end{aligned}
&=0, & \quad  &\forall w_1 \in V^1_{\ms,k},  \\
(\kappa \nabla (v_2-R^2_{\ms,k} v_2), \nabla w_2)&=0, & &\forall w_2\in V^2_{\ms,k}.\label{mskritz2}\\
(\sigma(\tilde R^K_{\f,k}v_2):\varepsilon(w))-(\alpha R^2_{\ms,k}v_2, \nabla \cdot w)_K&=0,& &\forall w \in V^1_\f(w_k(K)). \label{locstationarycorr}
\end{alignat}
Note that the Ritz projection is stable in the sense that
\begin{align}\label{locmsstable}
\|\tilde R^1_{\ms,k}(v_1,v_2)\|_{H^1}\leq C(\|v_1\|_{H^1} + \|v_2\|_{H^1}), \quad \|R^2_{\ms,k}v_2\|_{H^1}\leq C\|v_2\|_{H^1}.
\end{align}

The following two lemmas give a bound on the error introduced by the localization. 
\begin{mylemma}\label{locexp}
	For all $(v_1,v_2)\in V^1_h\times V^2_h$, there exists $\xi \in (0,1)$, such that
	\begin{align}
	\|R^1_{\f,k}v_1 - R^1_\f v_1\|^2_{H^1} &\leq Ck^d\xi^{2k}\sum_{K\in \mathcal T_H}\|R^{K,1}_\f v_1\|^2_{H^1}, \label{locexp1} \\
	\|R^2_{\f,k}v_2 - R^2_\f v_2\|^2_{H^1} &\leq Ck^d\xi^{2k}\sum_{K\in \mathcal T_H}\|R^{K,2}_\f v_2\|^2_{H^1}, \label{locexp2} \\
	\|\tilde R_{\f,k} v_2- \tilde R_\f v_2\|^2_{H^1} &\leq Ck^d\xi^{2k}\sum_{K\in \mathcal T_H}\|\tilde R^K_\f v_2\|^2_{H^1}.\label{locexp3}
	\end{align}
\end{mylemma}
The bounds \eqref{locexp1}-\eqref{locexp2} are direct results from \cite{HP16}, while \eqref{locexp3} follows by a slight modification of the right hand side. We omit the proof here. 

The next lemma gives a bound for the localized Ritz projection.

\begin{mylemma}\label{locstationaryconv}
	For all $(v_1,v_2)\in V^1_h\times V^2_h$ there exist $\xi \in (0,1)$ such that
	\begin{align}
	\label{locstationaryconv1}\|v_1 - \tilde R^1_{\ms,k}(v_1,v_2) \|_{H^1} &\leq C(H+k^{d/2}\xi^k)(\|\mathcal A_1(v_1,v_2)\|+ \|v_2\|_{H^1}), \\
	\|v_2 - R^2_{\ms,k} v_2\|_{H^1} &\leq C(H+k^{d/2}\xi^k)\|\mathcal A_2 v_2\|. \label{locstationaryconv2}
	\end{align}
\end{mylemma}
\begin{proof}
	It follows from \cite{Henning14} that \eqref{locstationaryconv2} holds. To prove \eqref{locstationaryconv1} we let $v_H\in V^1_H$ and $v_{H,k}\in V^1_H$ be elements such that
	\begin{align*}
	R^1_\ms(v_1,v_2) = v_H - R^1_\f v_H, \quad R^1_{\ms,k}(v_1,v_2) = v_{H,k} - R^1_{\f,k} v_{H,k}. 
	\end{align*}
	Define $e:= v_1-\tilde R^1_{\ms,k}(v_1,v_2)$. From \eqref{mskritz1}-\eqref{mskritz2} we get have the following identity for any $z\in V^1_{\ms,k}$
	\begin{align*}
	(\sigma(e):&\varepsilon(e)) - (\alpha(v_2-R^2_{\ms,k}v_2),\nabla\cdot e)
	\\&=(\sigma(e):\varepsilon(v_1 - z - \tilde R_{\f,k} v_1)) - (\alpha(v_2-R^2_{\ms,k}v_2),\nabla\cdot(v_1 - z- \tilde R_{\f,k} v_2)).
	\end{align*}
	Using this with $z=v_H - R^1_{\f,k}v_H \in V^1_{\ms,k}$ we get
	\begin{align*}
	c_\sigma \|e\|^2_{H^1} \leq (\sigma(e):\varepsilon(e)) &=  (\sigma(e):\varepsilon(v_1 -v_H - R^1_{\f,k}v_H- \tilde R_\f v_1)) \\&\quad- (\alpha(v_2-R^2_{\ms,k}v_2),\nabla\cdot(v_1 -v_H - R^1_{\f,k}v_H- \tilde R_{\f,k} v_2))\\
	&\quad +(\alpha(v_2-R^2_{\ms,k}v_2),\nabla\cdot e).
	\end{align*}
	Now, using Cauchy-Schwarz and Young's inequality we get
	\begin{align*}
	\|e\|^2_{H^1} &\leq C(\|v_1 -v_H - R^1_{\f,k}v_H- \tilde R_{\f,k} v_2\|^2_{H^1} +\|v_2-R^2_{\ms,k}v_2\|^2),
	\end{align*} 
	where the last term is bounded in \eqref{locstationaryconv2}. For the first term we get
	\begin{align*}
	&\|v_1 -v_H - R^1_{\f,k}v_H- \tilde R_{\f,k} v_2\|_{H^1} \\
	&\quad\leq \|v_1-(v_H - R^1_{\f}v_H + \tilde R_\f v_2)\|_{H^1} + \|R^1_\f v_H - R^1_{\f,k}v_H\|_{H^1} + \|\tilde R_\f v_2- \tilde R_{\f,k} v_2\|^2_{H^1}\\
	&\quad\leq \|v_1-\tilde R^1_\ms(v_1,v_2)\|_{H^1} + \|R^1_\f v_H - R^1_{\f,k}v_H\|_{H^1} + \|\tilde R_\f v_2- \tilde R_{\f,k} v_2\|_{H^1},
	\end{align*} 
	where the first term on the right hand side is bounded in Lemma~\ref{stationaryconv}. For the second term we use Lemma~\ref{locexp} to get
	\begin{align*}
	\|R^1_\f v_H - R^1_{\f,k}v_H\|^2_{H^1} &\leq Ck^d\xi^{2k}\sum_{K\in \mathcal T_H}\|R^{K,1}_\f v_H\|^2_{H^1} \leq Ck^d\xi^{2k}\sum_{K\in \mathcal T_H}\|v_H\|^2_{H^1(K)} \\&= Ck^d\xi^{2k}\|v_H\|^2_{H^1} = Ck^d\xi^{2k}\|I_H(v_H-R^1_\f v_H)\|^2_{H^1} \\&= Ck^d\xi^{2k}\|I_H R^1_\ms(v_1,v_2)\|^2_{H^1} \leq Ck^d\xi^{2k}\| R^1_\ms(v_1,v_2)\|^2_{H^1}.
	\end{align*}
	We can bound this further by using \eqref{msstable} and \eqref{A1operator}, such that
	\begin{align*}
	\|R^1_\ms(v_1,v_2)\|_{H^1}\leq C(\|v_1\|_{H^1} + \|v_2\|_{H^1}) \leq C(\|\mathcal A_1(v_1,v_2)\| + \|v_2\|_{H^1}). 
	\end{align*}
	
	Similar arguments, using Lemma~\ref{locexp} and \eqref{locstationarycorr}, prove 
	\begin{align*}
	\|\tilde R_\f v_2- \tilde R_{\f,k} v_2\|_{H^1} \leq Ck^{d/2}\xi^k\|v_2\|_{H^1},
	\end{align*} and \eqref{locstationaryconv1} follows.
\end{proof}

\begin{myremark}\label{remark-sizek}
	To preserve linear convergence, the localization parameter $k$ should be chosen such that $k=c \log(H^{-1})$ for some constant $c$. With this choice of $k$ we get $k^{d/2}\xi^k\sim H$ and we get linear convergence in Lemma~\ref{locstationaryconv}.
\end{myremark} 

We note that the orthogonality relation \eqref{orthog} does not hold when $V^1_{\ms}$ is replaced by $V^1_{\ms,k}$. However, we have that $V^1_{\ms,k}$ and $V^1_\f$ are almost orthogonal in the sense that
\begin{align}\label{almostorthog}
(\sigma(v):\varepsilon(w)) \leq Ck^{d/2}\xi^k\|v\|_{H^1}\|w\|_{H^1}, \quad \forall v\in V^1_{\ms,k}, \, w\in V^1_\f.
\end{align}
To prove this, note that $v=v_{H,k} - R^1_{\f,k}v_{H,k}$ for some $v_{H,k} \in V^1_H$, and 
\begin{align*}
(\sigma(v):\varepsilon(w)) &= (\sigma(v_{H,k} - R^1_\f v_{H,k}):\varepsilon(w)) + (\sigma(R^1_\f v_{H,k}- R^1_{\f,k} v_{H,k}):\varepsilon(w)) \\&= (\sigma(R^1_\f v_{H,k}- R^1_{\f,k} v_{H,k}):\varepsilon(w)) \leq C_\sigma\|R^1_\f v_{H,k}- R^1_{\f,k} v_{H,k}\|_{H^1}\|w\|_{H^1},
\end{align*}
where we have used that $v_{H,k}- R^1_\f v_{H,k} \in V^1_\ms$ and the orthogonality \eqref{orthog}. Due to Lemma~\ref{locexp} we now have 
\begin{align*}
\|R^1_\f v_{H,k}- R^1_{\f,k} v_{H,k}\|^2_{H^1} &\leq Ck^d\xi^{2k}\sum_{K\in \mathcal T_H}\|R^{K,1}_\f v_{H,k}\|^2_{H^1} \leq Ck^d\xi^{2k}\sum_{K\in \mathcal T_H}\|v_{H,k}\|^2_{H^1(K)} \\&= Ck^d\xi^{2k}\|v_{H,k}\|^2_{H^1} = Ck^d\xi^{2k}\|I_H(v_{H,k}-R^1_{\f,k} v_{H,k})\|^2_{H^1} \\&= Ck^d\xi^{2k}\|I_H v\|^2_{H^1} \leq Ck^d\xi^{2k}\|v\|^2_{H^1},
\end{align*}
and \eqref{almostorthog} follows.

\subsection{Time-dependent problem}
A localized version of \eqref{gfem1}-\eqref{timecorr} is now defined by replacing $V^1_\ms$ with $V^1_{\ms,k}$ and $V^2_\ms$ with $V^2_{\ms,k}$. The method thus reads; for $n\in \{1,...,N\}$ find \begin{align*}
\tilde u^n_{\ms,k}=u^n_{\ms,k} + \sum_{K\in \mathcal T_H}u^{n,K}_{\f,k}, \ \text{with } u^n_{\ms,k}\in V^1_{\ms,k}, \, u^{n,K}_{\f,k}\in V^1_\f(\omega_k(K)),
\end{align*}
and $\theta^n_{\ms,k} \in V^2_{\ms,k}$, such that
\begin{alignat}{2}
(\sigma(\tilde u^n_{\ms,k}):\varepsilon(v_1))- (\alpha \theta^n_{\ms,k}, \nabla \cdot v_1) &= (f^n,v_1),& \quad &\forall v_1 \in V^1_{\ms,k},\label{locgfem1}\\
\begin{aligned}[b](\ddt\theta^n_{\ms,k},v_2) + (\kappa \nabla &\theta^n_{\ms,k}, \nabla v_2) \\&+ (\alpha \nabla \cdot \ddt \tilde u^n_{\ms,k},v_2)\end{aligned}&= (g^n,v_2),& &\forall v_2 \in V^2_{\ms,k} \label{locgfem2},\\
(\sigma(u^{n,K}_{\f,k}):\varepsilon(w_1)) - (\alpha \theta^n_{\ms,k}, \nabla\cdot w_1)_K&=0,& &\forall w_1 \in V^1_\f(\omega_k(K)).\label{loctimecorr}
\end{alignat}
where $\theta^0_{\ms,k}=R^2_{\ms,k} \theta^0_h$. Furthermore, we define $\tilde u^0_{\ms,k}=u^0_{\ms,k} + \sum_{K\in \mathcal T_H}u^{0,K}_{\f,k}$, where $u^{0,K}_{\f,k}\in V^1_\f(\omega_k(K))$ is defined by \eqref{loctimecorr} for $n=0$ and $u^0_{\ms,k}\in V^1_\ms$ such that
\begin{align}\label{locu0ms}
(\sigma(\tilde u^0_{\ms,k}):\varepsilon(v_1)) - (\alpha \theta^0_{\ms,k}, \nabla\cdot v_1)=(f^0,v_1), \quad \forall v_1 \in V^1_{\ms,k}.
\end{align}
We also define $u^n_{\f,k}:=\sum_{K\in \mathcal T_H}u^{n,K}_{\f,k}$. Note that for $u^n_\f$ we have due to \eqref{timecorr}
\begin{align*}
(\sigma(u^n_\f):\varepsilon(w_1))-(\alpha \theta^n_\ms,\nabla\cdot w_1)=0, \quad \forall w_1 \in V^1_\f.
\end{align*}
For the localized version $u^n_{\f,k}$ this relation is not true. Instead, we prove the following lemma.
\begin{mylemma}\label{lemmacorr}
	For $w_1\in V^1_\f$, it holds that
	\begin{align*}
	|(\sigma(u^n_{\f,k}):\varepsilon(w_1))-(\alpha \theta^n_{\ms,k},\nabla\cdot w_1)|\leq Ck^{d/2}\xi^k\|\theta^n_{\ms,k}\|\|w_1\|_{H^1}.
	\end{align*}
\end{mylemma}
\begin{proof}
	Note that from \eqref{loctimecorr} we have
	\begin{align}
	(\sigma(u^{n,K}_{\f,k}):\varepsilon(w_1)) - (\alpha \theta^n_{\ms,k}, \nabla\cdot w_1)_K=0, \quad \forall w_1 \in V^1_\f(\omega_k(K)).\label{lochelpcorr}
	\end{align}
	This equation can be viewed as the localization of the following problem. Find $z^n_\f\in V^1_\f$, such that 
	\begin{align}
	(\sigma(z^{n}_{\f}):\varepsilon(w_1)) - (\alpha \theta^n_{\ms,k}, \nabla\cdot w_1)=0, \quad \forall w_1 \in V^1_\f.\label{helpcorr}
	\end{align}
	Now, \cite[Lemma 4.4]{HP16} gives the bound
	\begin{align*}
	\|z^n_\f-u^n_{\f,k}\|^2_{H^1} \leq Ck^{d}\xi^{2k}\sum_{K\in\mathcal T_H}\|z^{n,K}_\f\|^2_{H^1}
	\end{align*}
	where $z^n_\f=\sum_{K \in \mathcal T_H}z^{n,K}_\f$ such that
	\begin{align*}
	(\sigma(z^{n,K}_{\f}):\varepsilon(w_1)) - (\alpha \theta^n_{\ms,k}, \nabla\cdot w_1)_K=0, \quad\forall w_1 \in V^1_\f.
	\end{align*}
	Using this we derive the bound
	\begin{align}
	\|z^n_\f-u^n_{\f,k}\|^2_{H^1}&\leq Ck^{d}\xi^{2k}\sum_{K\in\mathcal T_H}\|z^{n,K}_\f\|^2_{H^1}\leq Ck^{d}\xi^{2k}\sum_{K\in\mathcal T_H}\|\theta^n_{\ms,k}\|^2_{L_2(K)}\label{locbound}\\&=Ck^{d}\xi^{2k}\|\theta^n_{\ms,k}\|^2.\notag
	\end{align}
	Now, to prove the lemma we use \eqref{helpcorr} and Cauchy-Schwarz inequality to get
	\begin{align*}
	|(\sigma(u^n_{\f,k}):\varepsilon(w_1))-(\alpha \theta^n_{\ms,k},\nabla\cdot w_1)|&= |(\sigma(u^n_{\f,k}-z^n_\f):\varepsilon(w_1))| \\
	&\leq C_\sigma\|u^n_{\f,k}-z^n_\f\|_{H^1}\|w_1\|_{H^1}.
	\end{align*}
	Applying \eqref{locbound} finishes the proof.

\end{proof}
The proof can be modified slightly to show the following bound
\begin{align}\label{lemmacorr-ddt}
|(\sigma(\ddt u^n_{\f,k}):\varepsilon(w_1))-(\alpha \ddt\theta^n_{\ms,k},\nabla\cdot w_1)|\leq Ck^{d/2}\xi^k\|\ddt\theta^n_{\ms,k}\|\|w_1\|_{H^1}.
\end{align}
Also note that it follows, by choosing $w_1=u^n_{\f,k}$ and $w_1=\ddt u^n_{\f,k}$ respectively, that
\begin{align}\label{boundcorr}
\|u^n_{\f,k}\|_{H^1}\leq C\|\theta^n_{\ms,k}\|,\quad \|\ddt u^n_{\f,k}\|_{H^1}\leq C\|\ddt \theta^n_{\ms,k}\|.
\end{align}

To prove that \eqref{locgfem1}-\eqref{loctimecorr} is well posed, we need the following condition on the size of $H$. 
\begin{myassump} We make the following assumption on the size of $H$.
	\begin{enumerate}[label=(A\arabic*)] 
		\setcounter{enumi}{3}
		\item $H \leq \min\bigg(\frac{1}{4C_\textrm{co}},\frac{c_\sigma}{(C_\textrm{co}+C_\text{ort})}\bigg)$, where $C_\textrm{co}$ is the constant in Lemma~\mbox{\ref{lemmacorr}} and $C_\text{ort}$ is the constant in the almost orthogonal property \eqref{almostorthog}. \label{H-bound}
	\end{enumerate}
\end{myassump}

\begin{mylemma}
	Assuming \ref{H-bound} the problem \eqref{locgfem1}-\eqref{loctimecorr} is well-posed.
\end{mylemma}
\begin{proof}
	This proof is similar the proof of Lemma~\ref{wellposed}, but we need to account for the lack of orthogonality and the fact that \eqref{timecorr} is not satisfied. 
	
	Given $u^{n-1}_{\ms,k}$, $\theta^{n-1}_{\ms,k}$, and $u^{n-1}_{\f,k}=\sum_{K}u^{n-1,K}_{\f,k}$, the equations \eqref{locgfem1}-\eqref{loctimecorr} yields a square system, so it is sufficient to prove that the solution is unique. Choosing $v_1=u^n_{\ms,k}-u^{n-1}_{\ms,k}$ in \eqref{locgfem1} and $v_2=\tau\theta^n_{\ms,k}$ in \eqref{locgfem2} and adding the resulting equations we get 
	\begin{align*}
	(\sigma(u^n_{\ms,k}):\varepsilon(u^n_{\ms,k}-u^{n-1}_{\ms,k})) &+ (\sigma(u^n_{\f,k}):\varepsilon(u^n_{\ms,k}-u^{n-1}_{\ms,k}))+ \tau(\ddt \theta^n_{\ms,k},\theta^n_{\ms,k}) \\ &+ c_\kappa\tau\|\theta^n_{\ms,k}\|^2_{H^1} + (\alpha \nabla\cdot (u^n_{\f,k}-u^{n-1}_{\f,k}), \theta^n_{\ms,k})
	\\&\leq  (f^n,u^n_{\ms,k}-u^{n-1}_{\ms,k}) + \tau(g^n,\theta^n_{\ms,k}).
	\end{align*}
	Now, using \eqref{symmetry} and 
	\begin{align*}
	(f^n,u^n_{\ms,k}-u^{n-1}_{\ms,k}) \leq C\|f^n\|_{H^{-1}} + \frac{1}{2}(\sigma(u^n_{\ms,k}-u^{n-1}_{\ms,k}):\epsilon(u^n_{\ms,k}-u^{n-1}_{\ms,k})).
	\end{align*}
	together with the estimate $\tau(\ddt \theta^n_{\ms,k},\theta^n_{\ms,k})\geq \frac{1}{2}\|\theta^n_{\ms,k}\|^2-\frac{1}{2}\|\theta^{n-1}_{\ms,k}\|^2$, gives
	\begin{align*}
	\frac{c_\sigma}{2}\|u^n_{\ms,k}\|^2_{H^1} &+ \frac{1}{4}\|\theta^n_{\ms,k}\|^2 +  c_\kappa\tau\|\theta^n_{\ms,k}\|^2_{H^1} + (\sigma(u^n_{\f,k}):\epsilon(u^n_{\ms,k})) + (\alpha \nabla\cdot u^n_{\f,k}, \theta^n_{\ms,k}) \\
	&\leq  C\|f^n\|^2_{H^{-1}} + \frac{\tau}{2}\|g^n\|^2 + \frac{C_\sigma}{2}\|u^{n-1}_{\ms,k}\|^2_{H^1} + \frac{1}{2}\|\theta^{n-1}_{\ms,k}\|^2 \\ &\qquad + (\sigma(\tilde u^n_{\f,k}):\varepsilon(u^{n-1}_{\ms,k})) + (\alpha \nabla\cdot u^{n-1}_{\f,k}, \theta^n_{\ms,k}).
	\end{align*}
	Using Lemma~\ref{lemmacorr} we have
	\begin{align*}
	(\alpha \nabla\cdot u^n_{\f,k}, \theta^n_{\ms,k}) &= (\alpha \theta^n_{\ms,k}, \nabla\cdot u^n_{\f,k}) - (\sigma(u^n_{\f,k}):\epsilon(u^n_{\f,k})) + (\sigma(u^n_{\f,k}):\epsilon(u^n_{\f,k}))\\
	&\geq -|(\alpha \theta^n_{\ms,k}, \nabla\cdot u^n_{\f,k}) - (\sigma(u^n_{\f,k}):\epsilon(u^n_{\f,k}))| + c_\sigma\|u^n_{\f,k}\|^2_{H^1}\\
	&\geq -C_\textrm{co} k^{d/2}\xi^k\|u^n_{\f,k}\|_{H^1}\|\theta^n_{\ms,k}\|+ c_\sigma\|u^n_{\f,k}\|^2_{H^1}.
	\end{align*}
	and the almost orthogonal property \eqref{almostorthog} gives
	\begin{align*}
	|(\sigma(u^n_{\f,k}):\varepsilon(u^n_{\ms,k}))| \geq -C_\text{ort}k^{d/2}\xi^k\|u^n_{\f,k}\|_{H^1}\|u^n_{\ms,k}\|_{H^1}.
	\end{align*}
	Now, using that $k$ should be chosen such that linear convergence is obtained, see Remark~\ref{remark-sizek}, that is  $k^{d/2}\xi^k \sim H$, we conclude after using Young's inequality that
	\begin{align*}
	&(\frac{c_\sigma}{2}-\frac{C_\text{ort}H}{2})\|u^n_{\ms,k}\|^2_{H^1} + (\frac{1}{8}-\frac{C_\textrm{co}H}{2})\|\theta^n_{\ms,k}\|^2 +  c_\kappa\|\theta^n_{\ms,k}\|^2_{H^1} \\&\qquad+ (c_\sigma - \frac{(C_\textrm{co}+C_\text{ort})H}{2})\|u^n_{\f,k}\|^2_{H^1} \\
	&\quad \leq  C(\|f^n\|^2_{H^{-1}} + \tau\|g^n\|^2 + \|u^{n-1}_{\ms,k}\|^2_{H^1} + \|\theta^{n-1}_{\ms,k}\|^2 + \|u^{n-1}_{\f,k}\|^2_{H^1}),
	\end{align*}
	where assumption \ref{H-bound} guarantees that the coefficients are positive. Hence, a unique solution exists.
\end{proof}

\section{Error analysis}\label{sec:error}  
In this section we analyze the error of the generalized finite element method. The results are based on assumption \ref{H-bound}. In the analysis we utilize the following property, which is similar to Lemma~\ref{lemmacorr}.
\begin{mylemma}\label{lemmacorreta}
	Let $\tilde e^n_{\f,k}:=\tilde R_{\f,k}\theta^n_h-u^n_{\f,k}$ and $\eta^n_\theta:=R^2_{\ms,k}\theta^n_h-\theta^n_{\ms,k}$. Then, for $w_1\in V^1_\f$, it holds that
	\begin{align*}
	|(\sigma(\tilde e^n_{\f,k}):\varepsilon(w_1))-(\alpha \eta^n_\theta,\nabla\cdot w_1)|\leq Ck^{d/2}\xi^k\|\eta^n_\theta\|\|w_1\|_{H^1}.
	\end{align*}
\end{mylemma}
\begin{proof}
	The proof is similar to the proof of Lemma~\ref{lemmacorr}. We omit the details.
\end{proof}
This can be modified slightly to show the following bound
\begin{align}\label{lemmacorreta-ddt}
|(\sigma(\ddt \tilde e^n_{\f,k}):\varepsilon(w_1))-(\alpha \ddt\eta^n_\theta,\nabla\cdot w_1)|\leq Ck^{d/2}\xi^k\|\ddt\eta^n_\theta\|\|w_1\|_{H^1}.
\end{align}
Also note that it follows, by choosing $w_1=\tilde e^n_{\f,k}$ and $w_1=\ddt \tilde e^n_{\f,k}$ respectively, that
\begin{align}\label{boundcorreta}
\|\tilde e^n_{\f,k}\|_{H^1}\leq C\|\eta^n_\theta\|,\quad \|\ddt\tilde e^n_{\f,k}\|_{H^1}\leq C\|\ddt\eta^n_\theta\|.
\end{align}

\begin{mythm}\label{gfemerror}
	Assume that \ref{H-bound} holds. Let $\{u^n_h\}_{n=1}^N$ and $\{\theta^n_h\}_{n=1}^N$ be the solutions to \eqref{fem1}-\eqref{fem2} and $\{\tilde u^n_{\ms,k}\}_{n=1}^N$ and $\{\theta^n_{\ms,k}\}_{n=1}^N$ the solutions to \eqref{locgfem1}-\eqref{loctimecorr}. For $n \in \{1,...,N\}$ we have
	\begin{align*}
	\|u^n_h-\tilde u^n_{\ms,k}\|_{H^1} + \|\theta^n_h-\theta^n_{\ms,k}\|_{H^1}  &\leq C(H + k^{d/2}\xi^k)\big(\|g\|_{L_\infty(L_2)} +\|\dot g\|_{L_\infty(H^{-1})} \\&\quad+\|f\|_{L_\infty(L_2)} + \|\dot f\|_{L_\infty(L_2)} + \|\ddot f\|_{L_\infty(H^{-1})} \\&\quad+ t_n^{-1/2}\|\theta^0_h\|_{H^1}\big).
	\end{align*}
\end{mythm}


The proof of Theorem~\ref{gfemerror} is based on two lemmas.
%

\begin{mylemma}\label{gfemerror1}
	Assume that $\theta^0_h=0$ and \ref{H-bound} holds. Let $\{u^n_h\}_{n=1}^N$ and $\{\theta^n_h\}_{n=1}^N$ be the solutions to \eqref{fem1}-\eqref{fem2} and $\{\tilde u^n_{\ms,k}\}_{n=1}^N$ and $\{\theta^n_{\ms,k}\}_{n=1}^N$ the solutions to \eqref{locgfem1}-\eqref{loctimecorr}. For $n \in \{1,...,N\}$ we have
	\begin{align*}
	\|u^n_h-\tilde u^n_{\ms,k}\|_{H^1} + \|\theta^n_h-\theta^n_{\ms,k}\|_{H^1} &\leq C(H+k^{d/2}\xi^k)\big(\|g\|_{L_\infty(L_2)} + \|\dot g\|_{L_\infty(H^{-1})} \\&\quad+ \|f\|_{L_\infty(L_2)} + \|\dot f\|_{L_\infty(L_2)} + \|\ddot f\|_{L_\infty(H^{-1})}\big).
	\end{align*}
\end{mylemma} 
\begin{proof}
	We divide the error into the terms
	\begin{align*}
	u^n_h-\tilde u^n_{\ms,k} &= u^n_h-\tilde R^1_{\ms,k}(u^n_h,\theta^n_h) + \tilde R^1_{\ms,k}(u^n_h,\theta^n_h)  - \tilde u^n_{\ms,k} =: \tilde\rho_u^n + \tilde\eta_u^n,\\
	\theta^n_h-\theta^n_{\ms,k} &= \theta^n_h-R^2_{\ms,k} \theta^n_h + R^2_{\ms,k} \theta^n_h - \theta^n_{\ms,k} =: \rho_\theta^n + \eta_\theta^n.
	\end{align*}
	We also adopt the following notation
	\begin{align*}
	\tilde e^n_{\f,k}:=\tilde R_{\f,k} \theta^n_h-u^n_{\f,k}, \quad \eta^n_u:=\tilde \eta^n_u- \tilde e^n_{\f,k} = R^1_{\ms,k}(u^n_h,\theta^n_h)  - u^n_{\ms,k}.
	\end{align*}
	From \eqref{fem2} it follows that 
	\begin{align*}
	(\kappa \nabla \theta^n_h,\nabla v_2)=(g^n-\ddt\theta^n_h-\nabla \cdot \ddt u^n_h,v_2),\quad \forall v_2 \in V^2_h,
	\end{align*}
	so by Lemma~\ref{locstationaryconv} we have the bound 
	\begin{align*}
	\|\rho^n_\theta\|_{H^1}\leq C(H+k^{d/2}\xi^k)\|P^2_hg^n-\ddt\theta^n_h-\nabla \cdot \ddt u^n_h\|,
	\end{align*}
	where $P^2_h$ denotes the $L_2$-projection onto $V^2_h$. Theorem~\ref{femreg} now completes this bound.
	Similarly, \eqref{fem1} gives
	\begin{align*}
	(\sigma(u^n_h):\varepsilon(v_1))- (\alpha \theta^n_h, \nabla \cdot v_1) = (f^n,v_1),\quad \forall v_1 \in V^1_h,
	\end{align*}
	so, again, by Lemma~\ref{locstationaryconv} we get
	\begin{align*}
	\|\tilde \rho^n_u\|_{H^1} \leq C(H+k^{d/2}\xi^k)(\|f^n\| + \|\theta^n_h\|_{H^1}),
	\end{align*}
	which can be further bounded by using Theorem~\ref{femreg}. To bound $\tilde \eta^n_u$ and $\eta^n_\theta$ we note that for $v_1\in V^1_{\ms,k}$
	\begin{align}\label{eq:eta2}
	(\sigma(\tilde \eta^n_u):&\varepsilon(v_1)) - (\alpha \eta^n_\theta,\nabla \cdot v_1) \\&= (\sigma(\tilde R^1_{\ms,k}(u^n_h,\theta^n_h)):\varepsilon(v_1)) - (\alpha R^2_{\ms,k}\theta^n_h,\nabla \cdot v_1) - (f^n,v_1)\notag \\
	&= (\sigma(u^n_h):\varepsilon(v_1))  - (\alpha \theta^n_h,\nabla \cdot v_1) - (f^n,v_1)=0,\notag
	\end{align}
	where we have used the Ritz projection \eqref{mskritz1}, and the equations \eqref{fem1} and \eqref{locgfem1}. Similarly, for $v_2\in V^2_{\ms,k}$ we have
	\begin{align*}
	(\ddt &\eta^n_\theta, v_2) + (\kappa \nabla \eta^n_\theta, \nabla v_2) + (\alpha \nabla \cdot \ddt \tilde\eta^n_u,v_2) \\&= (\ddt R^2_{\ms,k}\theta^n_h,v_2) + (\kappa \nabla R^2_{\ms,k} \theta^n_h, \nabla v_2) + (\alpha \nabla \cdot \ddt \tilde R^1_{\ms,k}(u^n_h,\theta^n_h),v_2) -(g^n,v_2)\notag\\
	&=(-\ddt \rho^n_\theta, v_2) + (-\alpha \nabla \cdot \ddt \tilde\rho^n_u, v_2)\notag
	\end{align*}
	For simplicity, we denote $\rho^n:=\rho^n_\theta+\alpha \nabla \cdot \tilde\rho^n_u$ such that
	\begin{align}\label{eq:eta1}
	(\ddt \eta^n_\theta,v_2) + (\kappa \nabla \eta^n_\theta, \nabla v_2) + (\alpha \nabla \cdot \ddt \tilde\eta^n_u,v_2) =(-\ddt \rho^n, v_2), \quad \forall v_2 \in V^2_{\ms,k}
	\end{align}
	Now, choose $v_1=\ddt \eta^n_u$ and $v_2=\eta^n_\theta$ and add the resulting equations. Note that the coupling terms on the left hand side results in the term $(\alpha \nabla \cdot \ddt \tilde e^n_{\f,k}, \eta^n_\theta)$. 
	We conclude that
	\begin{align*}
	(\sigma(\tilde \eta^n_u):\varepsilon(\ddt \eta^n_u)) + (\ddt \eta^n_\theta,\eta^n_\theta) + (\kappa \nabla \eta^n_\theta, \nabla \eta^n_\theta) =(-\ddt \rho^n, \eta^n_\theta) - (\alpha \nabla\cdot \ddt \tilde e^n_{\f,k},\eta^n_\theta),
	\end{align*}
	and by splitting the first term
	\begin{align}\label{eq:rho}
	(\sigma(\eta^n_u):\varepsilon(\ddt \eta^n_u)) &+ (\ddt \eta^n_\theta,\eta^n_\theta) + (\kappa \nabla \eta^n_\theta, \nabla \eta^n_\theta) \\&=(-\ddt \rho^n, \eta^n_\theta)-(\sigma(\tilde e^n_{\f,k}):\varepsilon(\ddt \eta^n_u))-(\alpha \nabla\cdot \ddt \tilde e^n_{\f,k},\eta^n_\theta).\notag
	\end{align}
	Using Lemma~\ref{lemmacorreta} we can bound
	\begin{align}\label{alphaterm}
	-(\alpha \nabla\cdot \ddt \tilde e^n_{\f,k},\eta^n_\theta) &\leq |(\alpha \nabla\cdot \ddt \tilde e^n_{\f,k},\eta^n_\theta)-(\sigma(\tilde e^n_{\f,k}):\varepsilon(\ddt \tilde e^n_{\f,k}))| \\&\qquad- (\sigma(\tilde e^n_{\f,k}):\varepsilon(\ddt \tilde e^n_{\f,k}))\notag\\
	&\leq Ck^{d/2}\xi^k\|\ddt \tilde e^n_{\f,k}\|_{H^1}\|\eta^n_\theta\| - (\sigma(\tilde e^n_{\f,k}):\varepsilon(\ddt \tilde e^n_{\f,k})),\notag
	\end{align}
	and the almost orthogonal property \eqref{almostorthog} together with \eqref{boundcorreta} gives
	\begin{align}\label{sigmaterm}
	- (\sigma(\tilde e^n_{\f,k}):\varepsilon(\ddt \eta^n_u)) \leq Ck^{d/2}\xi^k\|\tilde e^n_{\f,k}\|_{H^1}\|\ddt \eta^n_u\|_{H^1}\leq Ck^{d/2}\xi^k\|\eta^n_{\theta}\|\|\ddt \eta^n_u\|_{H^1}.
	\end{align}
	Thus, multiplying \eqref{eq:rho} by $\tau$ and using Cauchy-Schwarz and Young's inequality we get
	\begin{align*}
	C\tau\|\eta^n_\theta\|^2_{H^1} &+ \frac{1}{2}(\|\eta^n_u\|^2_\sigma +\|\tilde e^n_{\f,k}\|^2_\sigma-\|\tilde \eta^{n-1}_u\|^2_\sigma -\|\tilde e^{n-1}_{\f,k}\|^2_\sigma) + \frac{1}{2}(\|\eta^n_\theta\|^2  - \|\eta^{n-1}_\theta\|^2)  \\
	&\leq  C\tau\|\ddt \rho^n\|^2_{H^{-1}} + C\tau k^{d/2}\xi^k\|\eta^n_\theta\|(\|\ddt \eta^n_u\|_{H^1} + \|\ddt \tilde e^n_{\f,k}\|_{H^1}),
	\end{align*}
	where $\|\eta^n_\theta\|\leq C\|\eta^n_\theta\|_{H^1}$ can be kicked to the left hand side. Summing over $n$ gives
	\begin{align*}
	& C\sum_{j=1}^n\tau\|\eta^j_\theta\|^2_{H^1} + \frac{1}{2}(\|\eta^n_u\|^2_\sigma + \|\tilde e^n_{\f,k}\|^2_\sigma) + \frac{1}{2}\|\eta^n_\theta\|^2  \\&\quad\leq \frac{c_\sigma}{2}\|\tilde \eta^0_u\|^2_{H^1}
	 +C\sum_{j=1}^n\tau(\|\ddt \rho^j\|^2_{H^{-1}} + k^{d}\xi^{2k}(\|\ddt \eta^j_u\|^2_{H^1} + \|\ddt \tilde e^j_{\f,k}\|^2_{H^1})),
	\end{align*}
	where we have used that $\eta^0_\theta=0$. Furthermore, we note that if $\theta^0_h=0$, then $\tilde R_{\f,k} \theta^0_h=0$ and $u^0_{\f,k}=0$. Hence, $e^0_{\f,k}=0$. From \eqref{locu0ms} and \eqref{u0} we have, if $\theta^0_h=\theta^0_{\ms,k}=0$, for $v_1\in V^1_{\ms,k}$,
	\begin{align*}
	(\sigma(u^0_{\ms,k}):\varepsilon(v_1)) = (f^0,v_1) = (\sigma(u^0_{h}):\varepsilon(v_1)) = (\sigma(R^1_{\ms,k}(u^0_{h},0)):\varepsilon(v_1)),
	\end{align*}
	so also $\eta^0_u=0$.
	
	To bound $\ddt \rho^j_\theta$ and $\alpha\nabla \cdot \ddt \tilde\rho^j_u$ we note that due to \eqref{fem1} and \eqref{u0}, $\ddt u^n_h$ and $\ddt \theta^n_h$ satisfy the equation 
	\begin{align*}
	(\sigma(\ddt u^n_h):\varepsilon(v_1))- (\alpha \ddt\theta^n_h, \nabla \cdot v_1) = (\ddt f^n,v_1), \quad \forall v_1 \in V^1_h.
	\end{align*}	
	Hence, by Lemma~\ref{locstationaryconv} and the Aubin-Nitsche duality argument we have
	\begin{align}\label{rhobound1}
	\|\ddt \rho^j_\theta\|_{H^{-1}} &\leq \|\ddt \rho^j_\theta\| \leq C(H+k^{d/2}\xi^k)\|\ddt \rho^j_\theta\|_{H^1} \leq C(H+k^{d/2}\xi^k)\|\ddt \theta^j_h\|_{H^1},
	\end{align}
	and for $\ddt \tilde \rho^j_u$ we get
	\begin{align}\label{rhobound2}
	\|\alpha\nabla &\cdot \ddt \tilde\rho^j_u\|_{H^{-1}} \\&\leq \alpha_2\|\nabla \cdot \ddt \tilde\rho^j_u\| \leq C\|\ddt \tilde\rho^j_u\|_{H^1} \leq C(H+k^{d/2}\xi^k)(\|\ddt f^j\| + \|\ddt \theta^j_h\|_{H^1}).\notag
	\end{align}
	Thus, using \eqref{sigma-bounds}, we arrive at the following bound
	\begin{align}\label{eq:eta-u}
	\sum_{j=1}^n\tau\|\eta^j_\theta\|^2_{H^1} + \|\eta^n_u\|^2_{H^1} &+  \|\tilde e^n_{\f,k}\|^2_{H^1} + \|\eta^n_\theta\|^2  \\&\leq C(H+k^{d/2}\xi^k)^2\sum_{j=1}^n\tau\big(\|\ddt\theta^j_h\|^2_{H^1} + \|\ddt f^j\|^2\big)\notag\\&\quad + Ck^{d}\xi^{2k}\sum_{j=1}^n \tau(\|\ddt \eta^j_u\|^2_{H^1} + \|\ddt \tilde e^j_{\f,k}\|^2_{H^1}),\notag
	\end{align}
	where we apply Theorem~\ref{femreg} to the first sum on the right hand side. If we can find an upper bound on $\sum_{j=1}^n\tau(\|\ddt\eta^j_u\|^2_{H^1} + \|\ddt\tilde e^j_{\f,k}\|^2)$, then \eqref{eq:eta-u} gives a bound for $\|\tilde \eta^n_u\|_{H^1}\leq \|\eta^n_u\|_{H^1} + \|\tilde e^n_{\f,k}\|_{H^1}$. This is done next, and we bound $\|\eta^n_\theta\|_{H^1}$ at the same time. For this purpose, we choose $v_2=\ddt \eta^n_\theta$ in \eqref{eq:eta1} and note that it follows from \eqref{eq:eta2} that
	\begin{align}\label{eq:eta3}
	(\sigma(\ddt\tilde \eta^n_u):&\varepsilon(\ddt\eta^n_u)) - (\alpha \ddt \eta^n_\theta,\nabla \cdot \ddt \eta^n_u) = 0.
	\end{align}
	This also holds for $n=1$ since $\eta^0_\theta=0$ and $\tilde\eta^0_u=0$. Thus, by adding the resulting equations, we have
	\begin{align*}
	c_\sigma\|\ddt \eta^n_u\|^2_{H^1} &+ \|\ddt \eta^n_\theta\|^2 + (\kappa \nabla \eta^n_\theta,\nabla \ddt \eta^n_\theta) \\
	&= (-\ddt \rho^n,\ddt \eta^n_\theta) - (\sigma(\ddt \tilde e^n_{\f,k}):\varepsilon(\ddt \eta^n_u))- (\alpha \nabla \cdot \ddt \tilde e^n_{\f,k},\ddt \eta^n_\theta)\\
	&\leq \|\ddt\rho^n\|\|\ddt \eta^n_\theta\| + C_\mathrm{ort}k^{d/2}\xi^k\|\ddt \tilde e^n_{\f,k}\|_{H^1}\|\ddt \eta^n_u\|_{H^1} - (\alpha \nabla \cdot \ddt \tilde e^n_{\f,k},\ddt \eta^n_\theta)
	\end{align*}
	where we have used \eqref{almostorthog}. For the last term we use Lemma~\ref{lemmacorreta} to achieve
	\begin{align*}
	- (\alpha \nabla \cdot \ddt \tilde e^n_{\f,k},\ddt \eta^n_\theta) \leq C_\mathrm{co}k^{d/2}\xi^k\|\ddt \tilde e^n_{\f,k}\|_{H^1}\|\ddt\eta^n_\theta\| - (\sigma(\ddt \tilde e^n_{\f,k}):\varepsilon(\ddt \tilde e^n_{\f,k})).
	\end{align*}
	Thus, we have
	\begin{align*}
	& c_\sigma(\|\ddt \eta^n_u\|^2_{H^1} + \|\ddt \tilde e^n_{\f,k}\|^2_{H^1}) + \|\ddt \eta^n_\theta\|^2 + (\kappa \nabla \eta^n_\theta,\nabla \ddt \eta^n_\theta) \\
	&\quad \leq \|\ddt\rho^n\|\|\ddt \eta^n_\theta\| + C_\mathrm{ort}k^{d/2}\xi^k\|\ddt \tilde e^n_{\f,k}\|_{H^1}\|\ddt \eta^n_u\|_{H^1} + C_\mathrm{co}k^{d/2}\xi^k\|\ddt \tilde e^n_{\f,k}\|_{H^1}\|\ddt\eta^n_\theta\|,
	\end{align*}
	and using Young's inequality we deduce 
	\begin{align*}
	(c_\sigma-\frac{C_\mathrm{ort}k^{d/2}\xi^k}{2})&\|\ddt \eta^n_u\|^2_{H^1} + (c_\sigma-\frac{(C_\mathrm{ort}+C_\mathrm{co})k^{d/2}\xi^k}{2})\|\ddt \tilde e^n_{\f,k}\|^2_{H^1}) \\&+ (\frac{1}{2}-\frac{C_\mathrm{co}k^{d/2}\xi^k}{2})\|\ddt \eta^n_\theta\|^2 + (\kappa \nabla \eta^n_\theta,\nabla \ddt \eta^n_\theta)	\leq C\|\ddt\rho^n\|^2,
	\end{align*}
	where assumption \ref{H-bound} guarantees that the coefficients are positive. Multiplying by $\tau$, using that $\tau(\kappa \nabla \eta^n_\theta,\nabla \ddt \eta^n_\theta)\geq 1/2(\|\eta^n_\theta\|_\kappa-\|\eta^{n-1}_\theta\|_\kappa)$, and summing over $n$ we derive
	\begin{align*}
	\sum_{j=1}^n \tau(\|\ddt\eta^j_u\|^2_{H^1} &+ \|\ddt \tilde e^j_{\f,k}\|^2_{H^1}+ \|\ddt \eta^j_\theta\|^2) + \|\eta^n_\theta\|^2_{H^1}\\&\leq C\sum_{j=1}^n \tau\|\ddt \rho^j\|^2 \leq C(H+k^{d/2}\xi^k)\sum_{j=1}^n \tau(\|\ddt f^j\|^2+\|\ddt\theta^j_h\|^2_{H^1}),
	\end{align*}
	where we have used that $\eta^0_\theta=0$, the bound \eqref{kappa-bounds}, and \eqref{rhobound1}-\eqref{rhobound2}. We can now apply Theorem~\ref{femreg}. Thus, the lemma follows for $\|\theta^n_h-\theta^n_{\ms,k}\|_{H^1}$. Moreover, this bounds the last terms in \eqref{eq:eta-u}, which completes the proof.
\end{proof}

\begin{mylemma}\label{gfemerror2}
	Assume that $f=0$ and $g=0$, and that \ref{H-bound} holds. Let $\{u^n_h\}_{n=1}^N$ and $\{\theta^n_h\}_{n=1}^N$ be the solutions to \eqref{fem1}-\eqref{fem2} and $\{\tilde u^n_{\ms,k}\}_{n=1}^N$ and $\{\theta^n_{\ms,k}\}_{n=1}^N$ be the solutions to \eqref{locgfem1}-\eqref{loctimecorr}. For $n \in \{1,...,N\}$ we have
	\begin{align}\label{eq:error2}
	\|u^n_h-\tilde u^n_{\ms,k}\|_{H^1} + t_n^{1/2}\|\theta^n_h-\theta^n_{\ms,k}\|_{H^1} \leq C(H+k^{d/2}\xi^{k})\|\theta^0_h\|_{H^1}.
	\end{align}
\end{mylemma}

\begin{proof}
	As in the proof of Lemma~\ref{gfemerror1} we split the error into two parts
	\begin{align*}
	u^n_h-\tilde u^n_{\ms,k} = \tilde\rho_u^n + \tilde\eta_u^n,\quad \theta^n_h-\theta^n_{\ms,k} = \rho_\theta^n + \eta_\theta^n,
	\end{align*}
	where Lemma~\ref{locstationaryconv} and Theorem~\ref{femreg} gives 
	\begin{align*}
	\|\rho^n_\theta\|_{H^1}&\leq C(H+k^{d/2}\xi^k)\|-\ddt \theta^n_h - \nabla\cdot \ddt u^n_h\| \leq C(H+k^{d/2}\xi^k)t_n^{-1/2}\|\theta^0_h\|_{H^1},\\
	\|\tilde \rho^n_u\|_{H^1}&\leq C(H+k^{d/2}\xi^k)\|\theta^n_h\|_{H^1}\leq C(H+k^{d/2}\xi^k)\|\theta^0_h\|_{H^1}.
	\end{align*}
	Now, note that \eqref{eq:eta1} and \eqref{eq:eta3} holds also when $f=0$ and $g=0$. In particular, \eqref{eq:eta3} holds also for $n=1$ due to the definition of $u^0_{\ms,k}$ and $u^0_h$ in \eqref{locu0ms} and \eqref{u0} respectively. By choosing $v_2=\ddt \eta^n_\theta$ and adding the resulting equations we derive
	\begin{align*}
	c_\sigma\|\ddt \eta^n_u\|^2_{H^1} + \|\ddt \eta^n_\theta\|^2 + (\kappa \nabla \eta^n_\theta,\nabla \ddt\eta^{n}_\theta) &+ (\sigma(\ddt \tilde e^n_{\f,k}):\epsilon(\ddt \eta^n_u)) \\&+ (\alpha \nabla\cdot \ddt \tilde e^n_{\f,k},\ddt \eta^n_\theta) \leq  \|\ddt \rho^n\|\|\ddt \eta^n_\theta\|.
	\end{align*}
	Recall $\rho^n=\rho^n_\theta + \alpha \nabla \cdot \tilde \rho^n_u$. As in the proof of Lemma~\ref{gfemerror1} we get from Lemma~\ref{boundcorreta}
	\begin{align*}
	(\alpha \nabla \cdot \ddt \tilde e^n_{\f,k},\ddt \eta^n_\theta) \geq -C_\mathrm{co}k^{d/2}\xi^k\|\ddt \tilde e^n_{\f,k}\|_{H^1}\|\ddt\eta^n_\theta\| + (\sigma(\ddt \tilde e^n_{\f,k}):\varepsilon(\ddt \tilde e^n_{\f,k})).
	\end{align*}
	and from \eqref{almostorthog}
	\begin{align*}
	(\sigma(\tilde e^n_{\f,k}):\varepsilon(\ddt \eta^n_u)) \geq -C_\mathrm{ort}k^{d/2}\xi^k\|\ddt \tilde e^n_{\f,k}\|_{H^1}\|\ddt \eta^n_u\|_{H^1}.
	\end{align*}
	Hence, we have
	\begin{align*}
	(c_\sigma-\frac{C_\mathrm{ort}k^{d/2}\xi^k}{2})&\|\ddt \eta^n_u\|^2_{H^1} + (c_\sigma-\frac{(C_\mathrm{ort}+C_\mathrm{co})k^{d/2}\xi^k}{2})\|\ddt \tilde e^n_{\f,k}\|^2_{H^1} \\&+(\frac{1}{2}-\frac{C_\mathrm{co}k^{d/2}\xi^k}{2})\|\ddt \eta^n_\theta\|^2 + (\kappa \nabla \eta^n_\theta,\nabla \ddt\eta^n_\theta) \leq  \|\ddt \rho^n\|^2,
	\end{align*}
	and assumption \ref{H-bound} guarantees that the coefficients are positive. Multiplying by $\tau t^2_n$, using that $\tau(\kappa \nabla \eta^n_\theta,\nabla \ddt\eta^n_\theta)\geq 1/2(\|\eta^n\|^2_\kappa-\|\eta^{n-1}\|^2_\kappa)$ and  $t_n^2-t_{n-1}^2\leq 3\tau t_{n-1}$, for $n\geq2$, now give
	\begin{align*}
	C\tau t_n^2&(\|\ddt \eta^n_u\|^2_{H^1} + \|\ddt \tilde e^n_{\f,k}\|^2_{H^1} +\|\ddt \eta^n_\theta\|^2) + \frac{t_n^2}{2}\|\eta^n_\theta\|^2_\kappa- \frac{t_{n-1}^2}{2}\|\eta^{n-1}_\theta\|^2_\kappa\\&\leq  C\tau t_n^2\|\ddt \rho^n\|^2  + C\tau t_{n-1}\|\eta^{n-1}_\theta\|^2_\kappa.
	\end{align*}
	Note that this inequality also holds for $n=1$, since $\eta^0_\theta=0$ (recall $\theta^0_{\ms,k}=R^2_{\ms,k}\theta^0_h$). Summing over $n$ gives and using \eqref{kappa-bounds}
	\begin{align}\label{etabounds}
	C\sum_{j=1}^n \tau t_j^2(\|\ddt \eta^j_u\|^2_{H^1} + \|\ddt \tilde e^j_{\f,k}\|^2_{H^1} &+ \|\ddt \eta^j_\theta\|^2) + c_\kappa t_n^2\|\eta^n_\theta\|^2_{H^1} \\&\leq C\sum_{j=1}^n\tau t^2_j\|\ddt \rho^j\|^2 + C\sum_{j=1}^{n-1}\tau t_j\| \eta^j_\theta\|^2_{H^1},\notag
	\end{align}
	and since $f^n=0$ and $g^n=0$, Lemma~\ref{locstationaryconv} and the Aubin-Nitsche trick as in \eqref{rhobound1} together with Theorem~\ref{femreg} give
	\begin{align}\label{ddtrhobound}
	\|\ddt \rho^j\|\leq \|\ddt \rho^j_\theta\| + \alpha_2\|\ddt \rho^j_u\|_{H^1} &\leq C(H+k^{d/2}\xi^k)(\|\ddt \theta^j_h\|_{H^1}+\|\nabla \cdot \ddt u^j_h\|)\\ &\leq C(H+k^{d/2}\xi^k)t_j^{-1}\|\theta^0_h\|_{H^1}.\notag
	\end{align}
	To bound the last sum on the right hand side in \eqref{etabounds} we choose $v_1=\ddt \eta^n_u$ and $v_2=\eta^n_\theta$ in \eqref{eq:eta1} and \eqref{eq:eta2} and add the resulting equations. This gives
	\begin{align*}
	(\sigma(\eta^n_u):\varepsilon(\ddt \eta^n_u)) &+ (\ddt \eta^n_\theta,\eta^n_\theta) + (\kappa \nabla \eta^n_\theta,\nabla  \eta^n_\theta) \\
	&= (-\ddt \rho^n,\eta^n_\theta) - (\sigma(\tilde e^n_{\f,k}):\varepsilon(\ddt \eta^n_u)) -(\alpha \nabla\cdot \ddt \tilde e^n_{\f,k},\eta^n_\theta),
	\end{align*}
	where the use of \eqref{alphaterm} and \eqref{sigmaterm} gives
	\begin{align*}
	(\sigma(\eta^n_u):\varepsilon(\ddt \eta^n_u)) & + (\sigma(\tilde e^n_{\f,k}):\varepsilon(\ddt \tilde e^n_{\f,k}))+ (\ddt \eta^n_\theta,\eta^n_\theta) + (\kappa \nabla \eta^n_\theta,\nabla  \eta^n_\theta) \\
	&\leq \|\ddt \rho^n\|\|\eta^n_\theta\| + Ck^{d/2}\xi^k\|\eta^n_\theta\|(\|\ddt \eta^n_u\|_{H^1} + \|\ddt \tilde e^n_{\f,k}\|_{H^1}).
	\end{align*}
	Multiplying by $\tau t_n$ and using that $t_n-t_{n-1}=\tau$ we get
	\begin{align*}
	C\tau & t_n\|\eta^n_\theta\|^2_{H^1} + \frac{t_n}{2}(\|\eta^n_u\|^2_\sigma +\|\tilde e^n_{\f,k}\|^2_\sigma)  -\frac{t_{n-1}}{2}(\|\eta^{n-1}_u\|^2_\sigma + \|\tilde e^{n-1}_{\f,k}\|^2_\sigma) \\&\qquad+ \frac{t_n}{2}\|\eta^n_\theta\|^2 -\frac{t_{n-1}}{2}\|\eta^{n-1}_\theta\|^2 \\ &\leq Ct_n\tau(\|\ddt \rho^n\|\|\eta^n_\theta\|+ k^{d/2}\xi^k\|\eta^n_\theta\|(\|\ddt \eta^n_u\|_{H^1} + \|\ddt \tilde e^n_{\f,k}\|_{H^1}) \\&\qquad+ C\tau(\|\eta^{n-1}_u\|^2_\sigma+ \|\tilde e^{n-1}_{\f,k}\|^2_\sigma +\|\eta^{n-1}_\theta\|^2) \\
	&\leq Ct^2_n\tau\|\ddt \rho^n\|^2 + C_y t^2_n\tau k^d\xi^{2k}(\|\ddt \eta^n_u\|^2_{H^1} + \|\ddt \tilde e^n_{\f,k}\|^2_{H^1}) \\&\qquad+ C\tau(\|\tilde\eta^{n-1}_u\|^2_\sigma +\|\tilde e^{n-1}_{\f,k}\|^2_\sigma + \|\eta^{n-1}_\theta\|^2 +\|\eta^n_\theta\|^2),
	\end{align*}
	where we have used Young's (weighted) inequality on the form, $\tau t_n ab \leq \tau t_n^2 a^2 + \tau b^2/4$, in the last step. For the second term we have used the inequality with an additional $C_y$, i.e. $\tau t_n ab \leq C_y\tau t_n^2 a^2 + (4C_y)^{-1}\tau b^2$. Note that $C_y $ can be made arbitrarily small. Summing over $n$ and using \eqref{sigma-bounds} now gives
	\begin{align}\label{eq:etat_n}
	C\sum_{j=1}^n\tau t_j&\|\eta^{j}_\theta\|^2_{H^1} + \frac{c_\sigma t_n}{2}(\|\eta^n_u\|^2_{H^1} + \|\tilde e^{n}_{\f,k}\|^2_{H^1}) + \frac{t_n}{2}\|\eta^n_\theta\|^2  \\&\leq C\sum_{j=1}^n\tau t^2_j\|\ddt \rho^j\|^2+ C_y k^{d}\xi^{2k}\sum_{j=1}^n\tau t^2_j(\|\ddt \eta^j_u\|^2_{H^1} + \|\ddt \tilde e^j_{\f,k}\|^2_{H^1}))\notag \\&\quad+ C\sum_{j=0}^n\tau(\|\eta^{j}_u\|^2_{H^1} + \|\tilde e^{j}_{\f,k}\|^2_{H^1}+ \|\eta^{j}_\theta\|^2).\notag
	\end{align}
	We can now use \eqref{etabounds} to deduce
	\begin{align*}
	\sum_{j=1}^n\tau t^2_j (\|\ddt \eta^j_u\|^2_{H^1} + \|\ddt \tilde e^j_{\f,k}\|^2_{H^1}) &\leq C\sum_{j=1}^n\tau t^2_j\|\ddt \rho^j\|^2 + C\sum_{j=1}^{n-1}\tau t_j\| \eta^j_\theta\|^2_{H^1}.
	\end{align*}
	Using this in \eqref{eq:etat_n} gives 
	\begin{align}\label{eq:etat_n2}
	C\sum_{j=1}^n\tau t_j&\|\eta^{j}_\theta\|^2_{H^1} + \frac{c_\sigma t_n}{2}(\|\eta^n_u\|^2_{H^1} + \|\tilde e^{n}_{\f,k}\|^2_{H^1}) + \frac{t_n}{2}\|\eta^n_\theta\|^2  \\&\leq C\sum_{j=1}^n\tau t^2_j\|\ddt \rho^j\|^2+ C_y k^{d}\xi^{2k}\sum_{j=1}^n\tau t_j\|\eta^{j}_\theta\|^2_{H^1} \notag \\&\quad+ C\sum_{j=0}^n\tau(\|\eta^{j}_u\|^2_{H^1} + \|\tilde e^{j}_{\f,k}\|^2_{H^1}+ \|\eta^{j}_\theta\|^2).\notag
	\end{align}
	Since $C_y$ now can be made arbitrarily small the term $C_y k^{d}\xi^{2k}\sum_{j=1}^n\tau t_j\|\eta^{j}_\theta\|^2_{H^1}$ can be moved to the left hand side. To estimate the last sum on the right hand side in \eqref{eq:etat_n2} we multiply \eqref{eq:eta1} by $\tau$ and sum over $n$ to get
	\begin{align}\label{eq:etasum}
	(\eta^n_\theta-\eta^0_\theta,v_2) + (\kappa \nabla \sum_{j=1}^n\tau\eta^j_\theta, \nabla v_2) + (\alpha \nabla \cdot \tilde\eta^n_u-\tilde\eta^0_u,v_2) =(-\rho^n+\rho^0, v_2),
	\end{align}
	where we note that $\eta^0_\theta=0$ and $\tilde \eta^0_u=0$. By choosing $v_1=\eta^n_u$ in \eqref{eq:eta2} and $v_2=\eta^n_\theta$ in \eqref{eq:etasum} and adding the resulting equations we get
	\begin{align*}
	& c_\sigma\|\eta^n_u\|^2_{H^1} + \|\eta^n_\theta\|^2 + (\kappa \sum_{j=1}^n\tau\nabla\eta^j_\theta, \nabla \eta^n_\theta) \\&\quad\leq \|-\rho^n+\rho^0\| \|\eta^n_\theta\| - (\sigma(\tilde e^n_{\f,k}):\varepsilon(\eta^n_u)) - (\alpha \nabla \cdot \tilde e^n_{\f,k},\eta^n_\theta). \notag\\&\quad
	\leq \|-\rho^n+\rho^0\| \|\eta^n_\theta\| + C_\text{ort}k^{d/2}\xi^k\|\tilde e^n_{\f,k}\|_{H^1}\|\eta^n_u\|_{H^1} + C_\text{co}k^{d/2}\xi^k\|\tilde e^n_{\f,k}\|_{H^1}\|\eta^n_\theta\|\notag \\&\qquad- c_\sigma\|\tilde e^n_{\f,k}\|^2_{H^1},\notag
	\end{align*}
	where we have used the almost orthogonal property \eqref{almostorthog} and Lemma~\ref{lemmacorr}. We conclude that
	\begin{align}\label{eq:etasum2}
	(c_\sigma-&\frac{C_\text{ort}k^{d/2}\xi^k}{2})\|\eta^n_u\|^2_{H^1} +(c_\sigma-\frac{(C_\text{ort}+C_\text{co})k^{d/2}\xi^k}{2})\|\tilde e^n_{\f,k}\|^2_{H^1}\\&+ (\frac{1}{2}-\frac{C_\text{co}k^{d/2}\xi^k}{2})\|\eta^n_\theta\|^2 + (\kappa \sum_{j=1}^n\tau\nabla\eta^j_\theta, \nabla \eta^n_\theta) \leq C\|-\rho^n+\rho^0\|^2,\notag 
	\end{align}
	and assumption \ref{H-bound} guarantees positive coefficients. Now, note that we have the bound
	\begin{align*}
	\bigg(\kappa \sum_{j=1}^n\tau\nabla\eta^j_\theta, \nabla \eta^n_\theta\bigg) &= \bigg(\kappa \sum_{j=1}^n\tau\nabla\eta^j_\theta, \ddt\bigg( \sum_{j=1}^n\tau \nabla \eta^j_\theta\bigg)\bigg) \\&\geq \frac{1}{2\tau} \bigg(\| \sum_{j=1}^n\tau\eta^j_\theta\|^2_\kappa-\| \sum_{j=1}^{n-1}\tau\eta^i_\theta\|^2_\kappa\bigg),
	\end{align*}
	with the convention that $\sum_{j=1}^0\tau\eta^j_\theta = 0$. Multiplying \eqref{eq:etasum2} by $\tau$, summing over $n$, and using \eqref{kappa-bounds} thus gives
	\begin{align}\label{sumetabounds}
	\sum_{j=1}^n\tau(\|\eta^j_u\|^2_{H^1} &+ \|\tilde e^j_{\f,k}\|^2_{H^1}+\|\eta^j_\theta\|^2) + \frac{c_\kappa}{2} \| \sum_{j=1}^n\tau\eta^j_\theta\|^2_{H^1}\\&\leq C\sum_{j=1}^n\tau\|-\rho^j+\rho^0\|^2\notag\\&\leq C(H+k^{d/2}\xi^k)^2\sum_{j=1}^n\tau \|\theta^0_h\|^2_{H^1}\leq C(H+k^{d/2}\xi^k)^2t_n \|\theta^0_h\|^2_{H^1}.\notag
	\end{align}
	Here we have used the Aubin-Nitsche duality argument, Lemma~\ref{locstationaryconv} and Lemma~\ref{femreg} to deduce
	\begin{align*}
	\|\rho^j\|&\leq \|\rho^j_\theta\| + C\|\rho^j_u\|_{H^1} \leq C(H+k^{d/2}\xi^k)(\|\rho^j_\theta\|_{H^1} + \|\theta^n_h\|_{H^1})\\&\leq C(H+k^{d/2}\xi^k)\|\theta^n_h\|_{H^1}\leq C(H+k^{d/2}\xi^k)\|\theta^0_h\|_{H^1}, \quad j \geq 0.\notag
	\end{align*}  
	Combining \eqref{etabounds}, \eqref{ddtrhobound}, \eqref{eq:etat_n2}, and \eqref{sumetabounds} we get
	\begin{align*}
	t_n^2\|\eta^n_\theta\|^2_{H^1} + t_n\|\eta^n_u\|^2_{H^1} + t_n\|\tilde e^n_{\f,k}\|^2_{H^1}\leq C(H+k^{d/2}\xi^k)^2t_n\|\theta^0_h\|^2_{H^1},
	\end{align*}
	which completes the proof.
\end{proof}

\begin{proof}[Proof of Theorem~\ref{gfemerror}]
	Since the problem is linear we can split the solution 
	\begin{align*}
	u^n_h=\bar u^n_h + \hat u^n_h, \quad \theta^n_h = \bar \theta^n_h + \hat \theta^n_h,
	\end{align*}
	where $\bar u^n_h$ and $\bar \theta^n_h$ solves \eqref{fem1}-\eqref{fem2} with $f=0$ and $g=0$ and $\hat u^n_h$ and $\hat \theta^n_h$ solves \eqref{fem1}-\eqref{fem2} with $\theta^0=0$. The theorem now follows by applying Lemma~\ref{gfemerror1} and Lemma~\ref{gfemerror2}.
\end{proof}

\section{Numerical examples}\label{sec:ex}
In this section we perform two numerical examples. For a discussion on how to implement the type of generalized finite element efficiently described in this paper we refer to \cite{Henning16}.

The first numerical example models a composite material which is preheated to a fix temperature and at time $t_0=0$ the piece is subject to a cool-down.

The domain is set to be the unit square $\Omega=[0,1]\times[0,1]$ and we assume that the temperature has a homogeneous Dirichlet boundary condition, that is $\Gamma^\theta_D=\partial \Omega$ and $\Gamma^\theta_N=\emptyset$. For the displacement we assume the bottom boundary to be fix and for the remaining part of the boundary we prescribe a homogeneous Neumann condition, that is $\Gamma^u_D=[0,1]\times {0}$ and $\Gamma^u_N = \partial \Omega \setminus \Gamma^u_D$. 

The composite is assumed to be built up according to Figure~\ref{material}. The white part in the figure denotes a background material and the black parts an insulated material. The black squares are of size $2^{-5}\times 2^{-5}$. We assume that the Lam\'{e} coefficients $\mu$ and $\lambda$ take the values $\mu_1$ and $\lambda_1$ on the insulated material, and $\mu_2$ and $\lambda_2$ on the background material. In this experiment we have set $\mu_1/\mu_2=10$ and $\lambda_1/\lambda_2=50$. Similarly, using subscript 1 for the insulated material and subscript 2 for the background material, we set $\alpha_1/\alpha_2=10$ and $\kappa=\kappa_i\cdot I$, for $i=1,2$, where $I$ is the $2$-dimensional identity matrix and $\kappa_1/\kappa_2=10$. Furthermore, we have chosen to set $f=[0, 0]^\intercal$ (no external body forces) and $g=-10$. 
\begin{figure}[h]
	\includegraphics[width=0.5\textwidth]{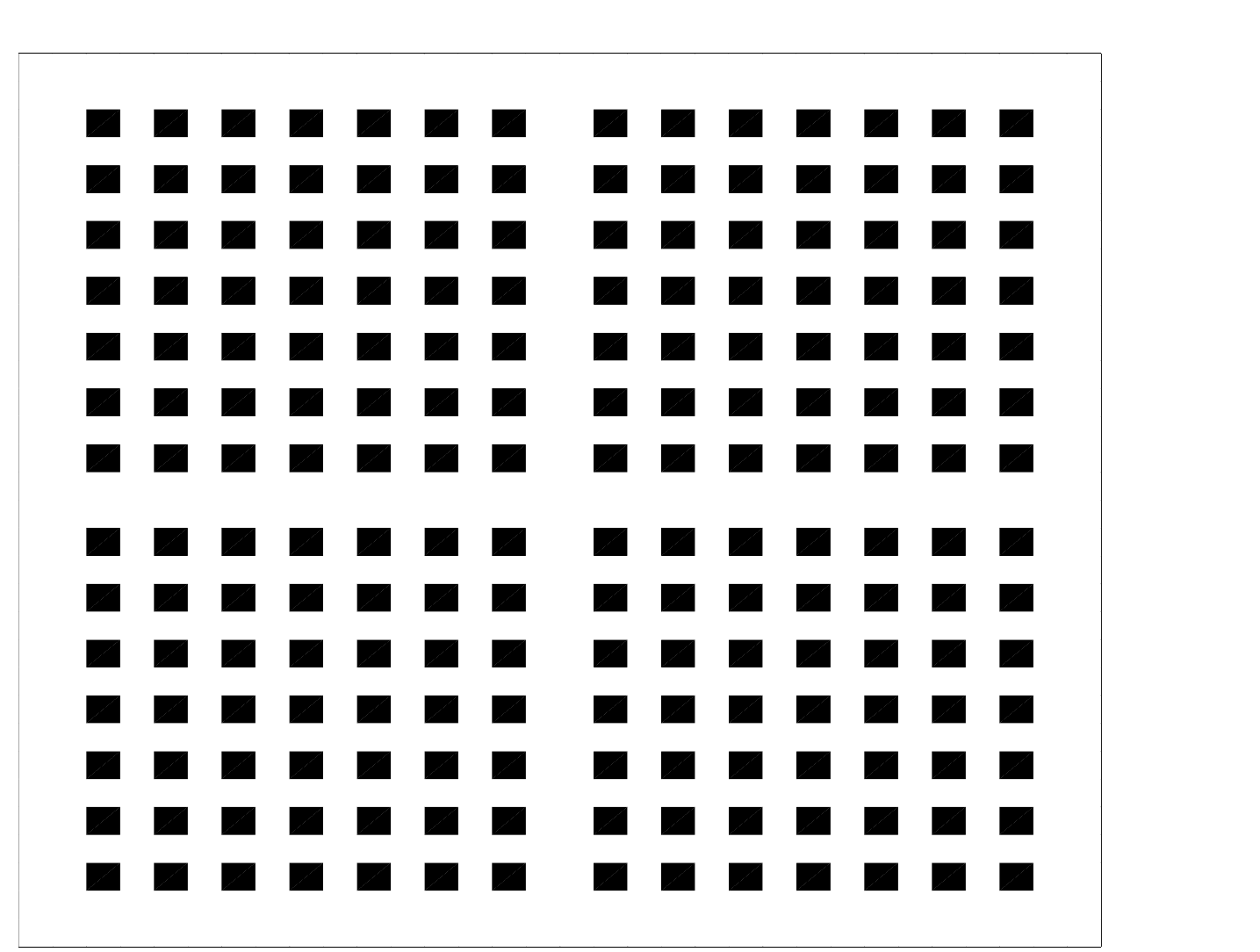}		
	\caption{Composite material on the unit square. One black square is of size $2^{-5}\times2^{-5}$.}\label{material}
\end{figure} 

The initial data must be zero on the boundary $\Gamma^\theta_D$, so we have chosen to put $\theta^0=500x(1-x)y(1-y)$ and $\theta^0_h$ to the $L_2$-projection of $\theta^0$ to $V^2_h$. For the generalized finite element solution we have chosen $\theta^0_{\ms,k}=R^2_{\ms,k}\theta^0_h$ and $\tilde u^0_{\ms,k}$ is given by \eqref{locu0ms}.

The domain is discretized using a uniform triangulation. The reference solution is computed on a mesh of $h=\sqrt{2}\cdot2^{-6}$ which resolves the fine parts (the black squares) in the material. The generalized finite element method (GFEM) in \eqref{locgfem1}-\eqref{loctimecorr} is computed for five decreasing values of the mesh size, namely, $H=\sqrt{2}\cdot2^{-1},\sqrt{2}\cdot2^{-2},...,\sqrt{2}\cdot2^{-5},$ with the patch sizes $k=1,1,2,2,3$. For comparison, we also compute the corresponding classical finite element (FEM) solution on the coarse meshes using continuous piecewise affine polynomials for both spaces (P1-P1). The solutions satisfies \eqref{fem1}-\eqref{fem2} with $h$ replaced by $H$ and are denoted $u^n_H$ and $\theta^n_H$ respectively for $n=1,...,N$. When computing these solutions we have evaluated the integrals exactly to avoid quadrature errors. 

We have chosen to set $T=1$ and $\tau=0.05$ for all values of $H$ and for the reference solution. The solutions are compared at the time point $N$.

Note that the implementation of the corrections $u^{n,K}_{\f,k}$ in \eqref{loctimecorr} given by
\begin{align*}
(\sigma(u^{n,K}_{\f,k}):\varepsilon(w_1)) - (\alpha \theta^n_{\ms,k}, \nabla\cdot w_1)_K=0, \quad \forall w_1 \in V^1_\f(\omega_k(K)),
\end{align*} 
should \textit{not} be computed explicitly at each time step. It is more efficient to compute $x^K_y$, given by
\begin{align*}
(\sigma(x^{K}_y):\varepsilon(w_1)) - (\alpha (\lambda^2_y-R^2_{\f,k}\lambda^2_y), \nabla\cdot w_1)_K=0, \quad \forall w_1 \in V^1_\f(\omega_k(K)),
\end{align*} 
where $\{(\cdot,y)\in \mathcal N:\lambda^2_y-R^2_{\f,k}\lambda^2_y\}$ is the basis for $V^2_{\ms,k}$. Now, since $\theta^n_{\ms,k}=\sum_y \beta^n_y (\lambda^2_y-R^2_{\f,k}\lambda^2_y)$, we have the identity
\begin{align*}
u^n_{\f,k} = \sum_K u^{n,K}_{\f,k} = \sum_K \sum_y \beta^n_y x^K_y.
\end{align*} 
With this approach, we only need to compute $x^K_y$ once before solving for the system \eqref{locgfem1}-\eqref{locgfem2} for $n=1,...,N$.  

The relative errors in the $H^1$-seminorm $\|\nabla \cdot\|$ are shown in Figure~\ref{conv}. The left graph shows the relative errors for the displacement, $\|\nabla(\tilde u^N_{\ms,k}-u^N_h)\|/\|\nabla u^N_h\|$ and $\|\nabla(u^N_{H}-u^N_h)\|/\|\nabla u^N_h\|$. The right graph shows the error for the temperature $\|\nabla(\theta^N_{\ms,k}-\theta^N_h)\|/\|\nabla \theta^N_h\|$ and $\|\nabla(\theta^N_{H}-\theta^N_h)\|/\|\nabla \theta^N_h\|$. As expected the generalized finite element shows convergence of optimal order and outperforms the classical finite element.   

\begin{figure}[h]
	\centering
	\begin{subfigure}[b]{0.48\textwidth}
		\includegraphics[width=\textwidth]{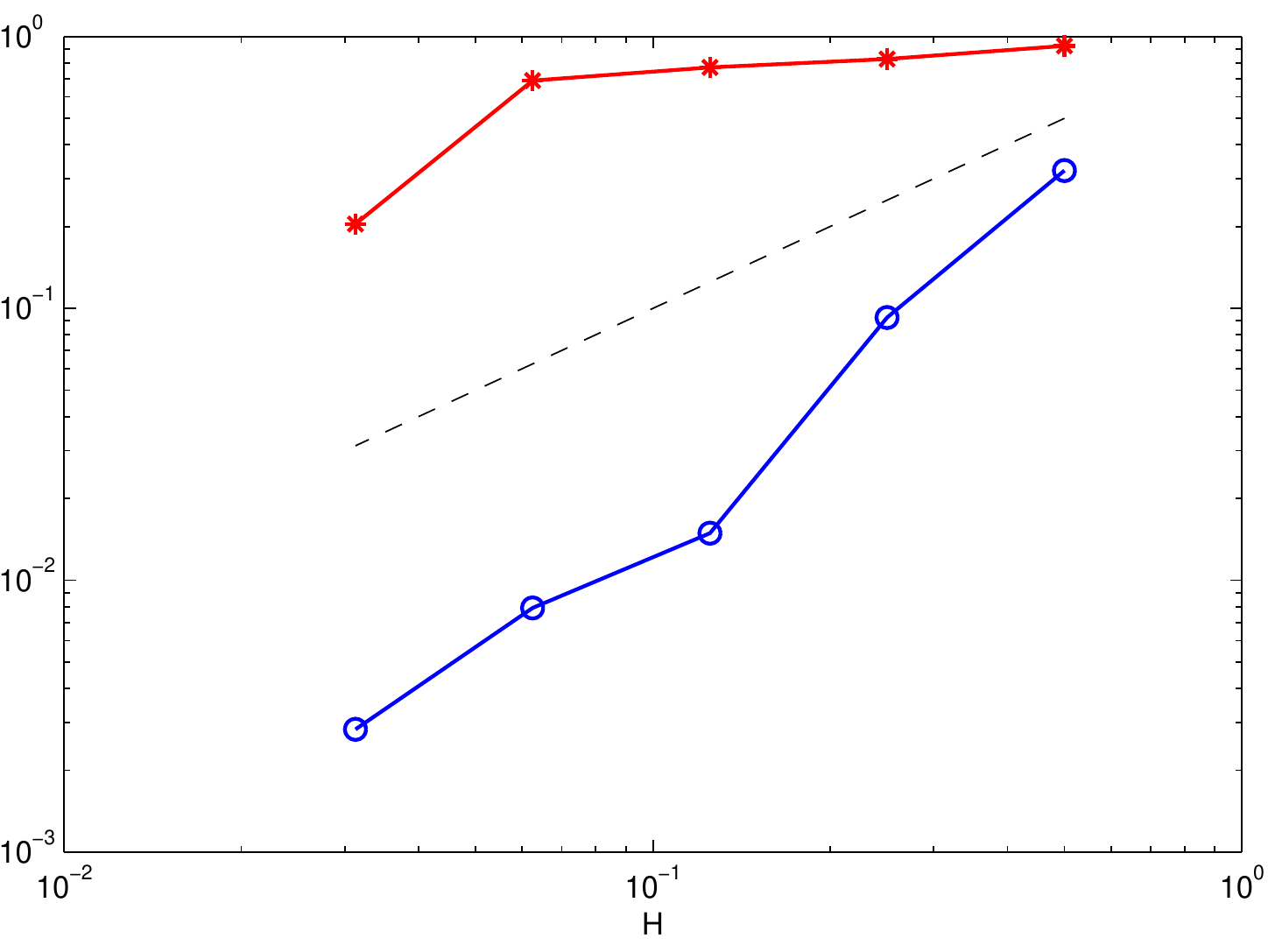}
		\caption{Displacement $u$}
	\end{subfigure}
	~
	\begin{subfigure}[b]{0.48\textwidth}
		\centering
		\includegraphics[width=\textwidth]{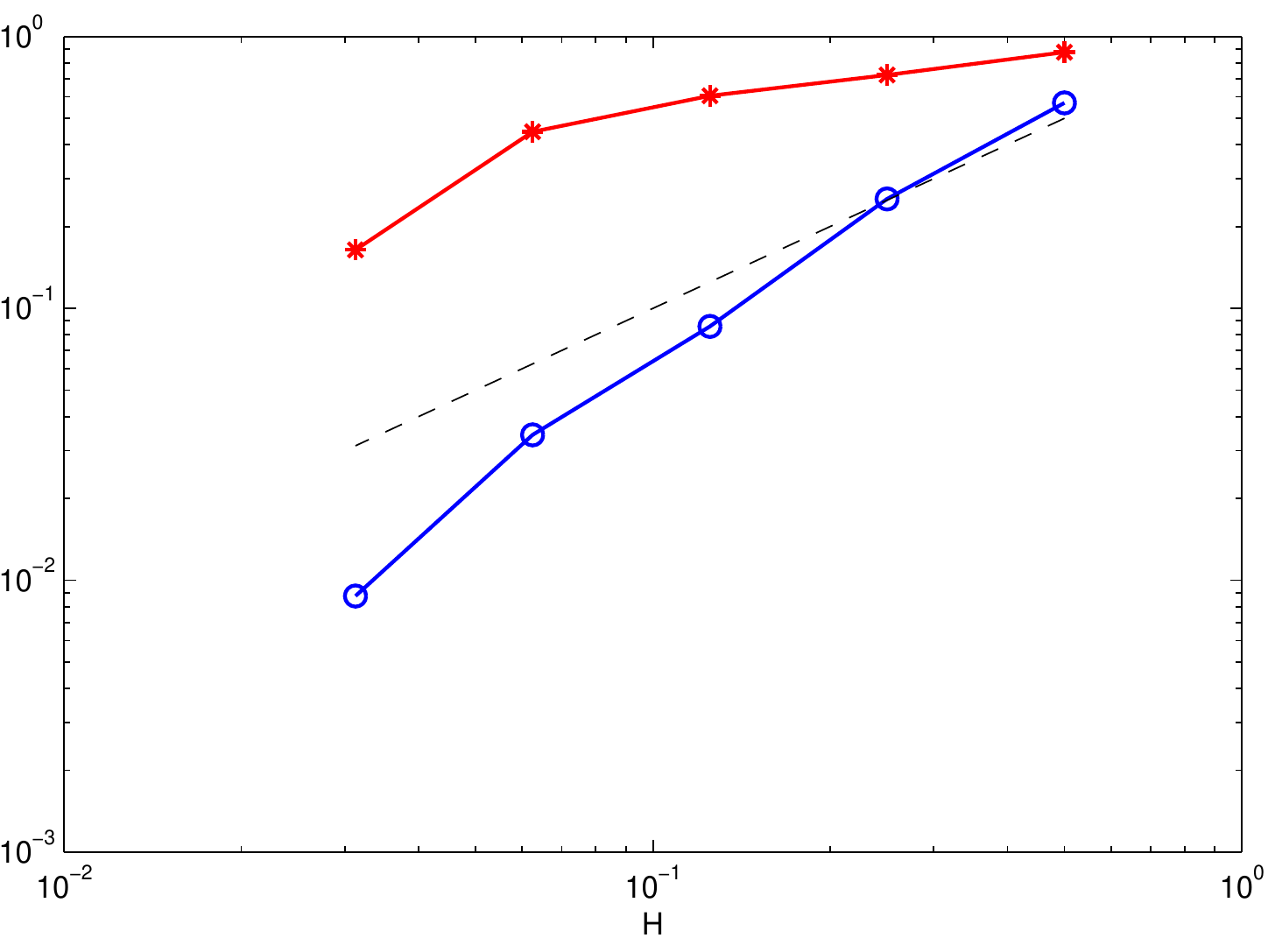}
		\caption{Temperature $\theta$} 
	\end{subfigure}
	\caption{Relative errors using GFEM (blue $\circ$) and P1-P1 FEM (red $\ast$) for the linear thermoelasticity problem plotted against the mesh size $H$. The dashed line is $H$.}\label{conv}
\end{figure} 

The second example shows the importance of the additional correction \eqref{loctimecorr}, which is designed to handle multiscale behavior in the coefficient $\alpha$. The computational domain, the spatial and the time discretization, and the patch sizes remain the same as in the first example. However, we let $\Gamma_D=\partial \Omega$ and $\Gamma_N=\emptyset$ in this case. 

To test the influence of $\alpha$ we let the other coefficients be constants, $\mu=\lambda=1$ and $\kappa=I$, where the $I$ is the $2$-dimensional identity matrix. The coefficient $\alpha$ takes values between $0.1$ and $10$ according to Figure~\ref{alpha-coeff}. The boxes are of size $2^{-5}\times 2^{-5}$ and, hence, the reference mesh of size $h=\sqrt{2}\cdot 2^{-6}$ is sufficiently small to resolve the variations in $\alpha$.  

\begin{figure}[h]
	\centering
	\includegraphics[width=0.5\textwidth]{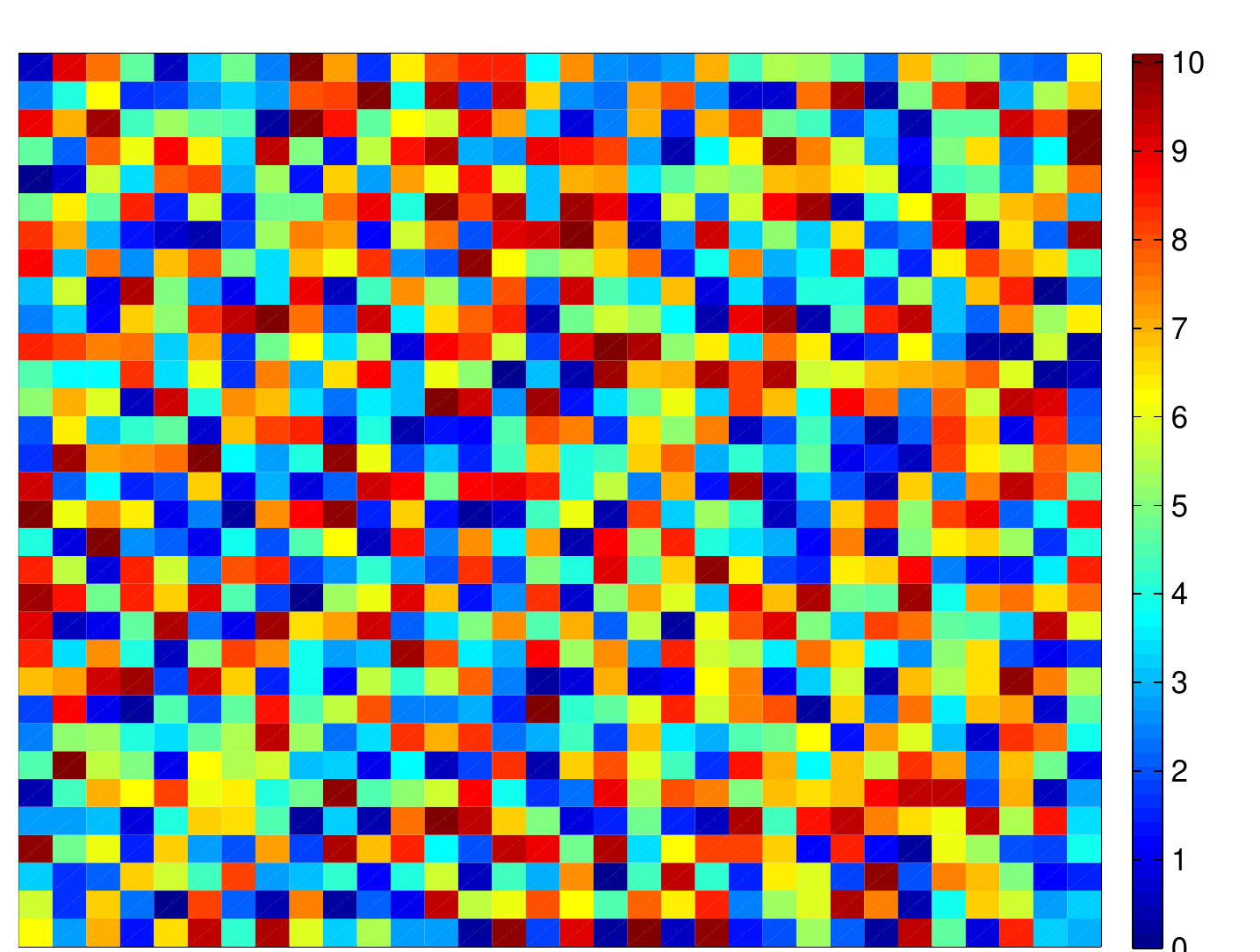}
	\caption{A plot of the coefficient $\alpha$.}\label{alpha-coeff}
\end{figure}

The initial data is set to $\theta^0=x(1-x)y(1-y)$ and $\theta^0_h$ is the $L_2$-projection of $\theta^0$ onto $V^2_h$. For the generalized finite element solution we have chosen $\theta^0_{\ms,k}=R^2_{\ms,k}\theta^0_h$ and $\tilde u^0_{\ms,k}$ is given by \eqref{locu0ms}, as in our first example. Furthermore, we have chosen to set $f=[1 \ 1]^\intercal$ and $g=10$. 

The generalized finite element method (GFEM) in \eqref{locgfem1}-\eqref{loctimecorr} is computed for the five decreasing values of the mesh size used in the first example. For comparison, we compute the generalized finite element without the additional correction on $u^n_{\ms,k}$. In this case the system \eqref{locgfem1}-\eqref{loctimecorr} simplifies to
\begin{alignat*}{2}
(\sigma(u^n_{\ms,k}):\varepsilon(v_1))- (\alpha \theta^n_{\ms,k}, \nabla \cdot v_1) &= (f^n,v_1),& \quad &\forall v_1 \in V^1_{\ms,k},\\
(\ddt\theta^n_{\ms,k},v_2) + (\kappa \nabla \theta^n_{\ms,k}, \nabla v_2)+ (\alpha \nabla \cdot \ddt u^n_{\ms,k},v_2)&= (g^n,v_2).& &\forall v_2 \in V^2_{\ms,k}
\end{alignat*}
The relative errors in the $H^1$-seminorm are shown in Figure~\ref{conv}. The graph shows the errors for the displacement with correction for $\alpha$, $\|\nabla(\tilde u^N_{\ms,k}-u^N_h)\|/\|\nabla u^N_h\|$ and the error without correction for $\alpha$ $\|\nabla(u^N_{\ms,k}-u^N_h)\|/\|\nabla u^N_h\|$. As expected the GFEM with correction for $\alpha$ shows convergence of optimal order and outperforms the GFEM without correction for $\alpha$. This is due to the fact that the constant in \eqref{locstationaryconv1} (and hence also the constant in Theorem~\ref{gfemerror}) depends on the variations in $\alpha$. 
	
\begin{figure}[h]
		\centering
		\includegraphics[width=0.5\textwidth]{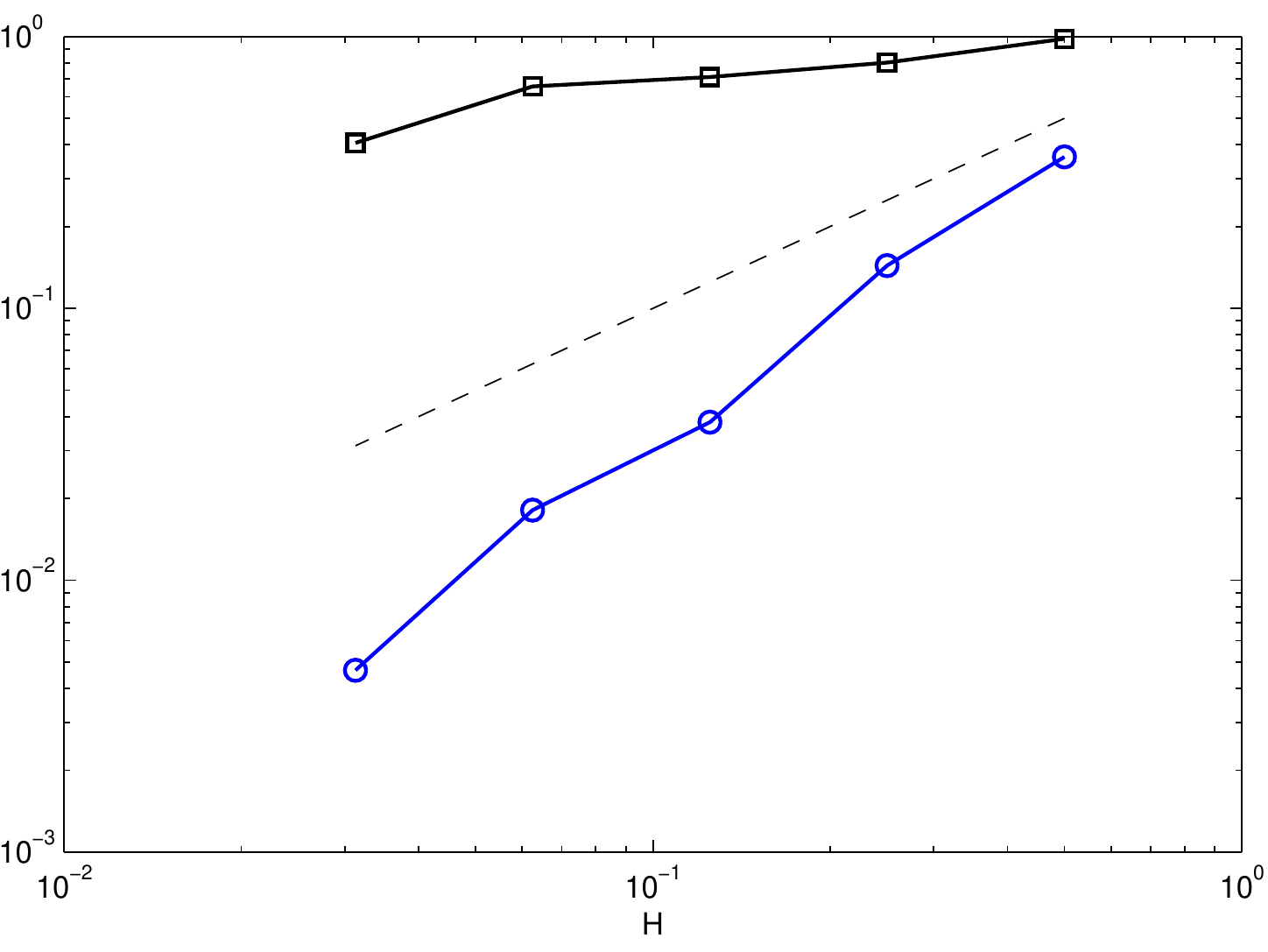}
	\caption{Relative errors for the displacement $u$ using GFEM with correction for $\alpha$ (blue $\circ$) and GFEM without correction for $\alpha$ (black $\square$) for the linear thermoelasticity problem plotted against the mesh size $H$. The dashed line is $H$.}\label{alphaconv}
\end{figure} 

\bibliographystyle{plain}
\bibliography{elasticity_ref}
\end{document}